\documentclass[12pt,english]{article}
\usepackage[T1]{fontenc}
\usepackage[latin9]{inputenc}
\usepackage{geometry}
\usepackage{babel}
\usepackage{amsthm}
\usepackage{amsmath}
\usepackage{amssymb}
\usepackage[unicode=true]
 {hyperref}
\usepackage{breakurl}

\makeatletter

\providecommand{\tabularnewline}{\\}

\numberwithin{equation}{section}
\numberwithin{figure}{section}
\theoremstyle{plain}
\newtheorem{thm}{\protect\theoremname}
  \theoremstyle{plain}
  \newtheorem{prop}[thm]{\protect\propositionname}
  \theoremstyle{plain}
  \newtheorem{cor}[thm]{\protect\corollaryname}

\usepackage{bm}
\usepackage{mathtools}

\usepackage{tikz}
\usetikzlibrary{arrows}

\DeclareMathOperator{\Mat}{Mat}
\DeclareMathOperator{\pk}{pk}
\DeclareMathOperator{\rpk}{rpk}
\DeclareMathOperator{\lrpk}{lrpk}
\DeclareMathOperator{\br}{br}
\DeclareMathOperator{\udr}{udr}
\DeclareMathOperator{\des}{des}
\DeclareMathOperator{\altdes}{altdes}
\DeclareMathOperator{\st}{st}
\DeclareMathOperator{\dasc}{dasc}
\DeclareMathOperator{\rdasc}{rdasc}
\DeclareMathOperator{\lrdasc}{lrdasc}
\DeclareMathOperator{\as}{as}

\DeclareMathOperator{\sech}{sech}

\usepackage{titlesec}
\titleformat{\section}
  {\normalfont\large\bfseries}{\thesection.}{.8ex}{} 
\titleformat{\subsection}
  {\normalfont\bfseries}{\thesubsection.}{.8ex}{} 

\makeatother

  \providecommand{\corollaryname}{Corollary}
  \providecommand{\propositionname}{Proposition}
\providecommand{\theoremname}{Theorem}

\begin{document}
\global\long\def\Mat{\operatorname{Mat}}

\global\long\def\pk{\operatorname{pk}}

\global\long\def\rpk{\operatorname{rpk}}

\global\long\def\lrpk{\operatorname{lrpk}}

\global\long\def\br{\operatorname{br}}

\global\long\def\udr{\operatorname{udr}}

\global\long\def\des{\operatorname{des}}

\global\long\def\altdes{\operatorname{altdes}}

\global\long\def\st{\operatorname{st}}

\global\long\def\sech{\operatorname{sech}}

\global\long\def\dasc{\operatorname{dasc}}

\global\long\def\rdasc{\operatorname{rdasc}}

\global\long\def\lrdasc{\operatorname{lrdasc}}

\global\long\def\as{\operatorname{as}}

\title{Counting permutations by runs}

\author{Yan Zhuang\\
Department of Mathematics\\
Brandeis University%
\thanks{MS 050, Waltham, MA 02453%
}\\
\texttt{\href{mailto:zhuangy@brandeis.edu}{zhuangy@brandeis.edu}}}
\maketitle
\begin{abstract}
In his Ph.D. thesis, Ira Gessel proved a reciprocity formula for noncommutative
symmetric functions which enables one to count words and permutations
with restrictions on the lengths of their increasing runs. We generalize
Gessel's theorem to allow for a much wider variety of restrictions
on increasing run lengths, and use it to complete the enumeration
of permutations with parity restrictions on peaks and valleys, and
to give a systematic method for obtaining generating functions for
permutation statistics that are expressible in terms of increasing
runs. Our methods can also be used to obtain analogous results for
alternating runs in permutations.
\end{abstract}
\textbf{\small{Keywords:}}{\small{ permutations, increasing runs,
peaks, valleys, noncommutative symmetric functions}}{\small \par}

\section{Introduction}

Given a set $A$, let $A^{*}$ be the set of all finite sequences
of elements of $A$, including the empty sequence. We call $A$ an
\textit{alphabet}, the elements of $A$ \textit{letters}, $A^{*}$
the \textit{free monoid} on $A$,%
\footnote{Indeed, $A^{*}$ is a free monoid generated by the letters in $A$,
under the operation of concatenation.%
} and the elements of $A^{*}$ \textit{words}. If we refer to words
without specifying an alphabet, then we take the alphabet to be $\mathbb{P}$,
the set of positive integers. Suppose that our alphabet $A$ is a
totally ordered set, such as $\mathbb{P}$. Then every word in $A^{*}$
can be uniquely decomposed into a sequence of maximal weakly increasing
consecutive subsequences, which we call \textit{increasing runs}.
For example, the increasing runs of $2142353$ are $2$, $14$, $235$,
and $3$. The notion of increasing runs clearly extends to permutations
as well.

As part of his 1977 Ph.D. thesis, Gessel proved a very general identity
involving noncommutative symmetric functions \cite[Theorem 5.2]{gessel-thesis}
from which we can obtain many generating functions for words and permutations
with restrictions on the lengths of their increasing runs. We call
this result the ``run theorem''. In this paper, we present a generalization
of Gessel's run theorem that will enable us to count words and permutations
with an even wider variety of restrictions on increasing run lengths.
Specifically, these restrictions are those which can be encoded by
a special type of digraph that we shall call a ``run network''.

The organization of this paper is as follows. In Section 2, we introduce
some preliminary definitions, state Gessel's run theorem, and present
our generalization of the run theorem. In Sections 3 and 4, we present
two separate applications of the generalized run theorem to permutation
enumeration.

Let $\pi=\pi_{1}\pi_{2}\cdots\pi_{n}$ be a permutation in $\mathfrak{S}_{n}$,
the set of permutations of $[n]=\{1,2,\dots,n\}$ (or more generally,
any sequence of $n$ distinct integers); such permutations are called
$n$-\textit{permutations}. We say that $i$ is a \textit{peak} of
$\pi$ if $\pi_{i-1}<\pi_{i}>\pi_{i+1}$ and that $i$ is a \textit{valley}
of $\pi$ if $\pi_{i-1}>\pi_{i}<\pi_{i+1}$. For example, given $\pi=5736214$,
its peaks are 2 and 4, and its valleys are 3 and 6.

In \cite{Gessel2014}, Gessel and Zhuang found the exponential generating
function 
\begin{equation}
\frac{3\sin\left(\frac{1}{2}x\right)+3\cosh\left(\frac{1}{2}\sqrt{3}x\right)}{3\cos\left(\frac{1}{2}x\right)-\sqrt{3}\sinh\left(\frac{1}{2}\sqrt{3}x\right)}\label{e-gz2014trig}
\end{equation}
for permutations with all peaks odd and all valleys even. Amazingly,
(\ref{e-gz2014trig}) can also be expressed as 
\begin{equation}
\left(1-E_{1}x+E_{3}\frac{x^{3}}{3!}-E_{4}\frac{x^{4}}{4!}+E_{6}\frac{x^{6}}{6!}-E_{7}\frac{x^{7}}{7!}+\cdots\right)^{-1}\label{e-gz2014}
\end{equation}
where the Euler numbers $E_{n}$ are defined by the identity $\sum_{n=0}^{\infty}E_{n}x^{n}/n!=\sec x+\tan x$,
which is reminiscent of David and Barton's \cite{David1962} generating
function 
\begin{equation}
\left(1-x+\frac{x^{3}}{3!}-\frac{x^{4}}{4!}+\frac{x^{6}}{6!}-\frac{x^{7}}{7!}+\cdots\right)^{-1}\label{e-davidbarton}
\end{equation}
for permutations with no increasing runs of length 3 or greater. The
authors explained the similarity between these two generating functions
by applying two different homomorphisms to an identity obtained by
the run theorem, which we review in Section 2.2 after introducing
the run theorem. 

The generating function (\ref{e-gz2014}) also counts permutations
with all peaks even and all valleys odd, because these permutations
are in bijection with permutations with all peaks odd and all valleys
even; the peaks (respectively, valleys) of $\pi=\pi_{1}\pi_{2}\cdots\pi_{n}$
are precisely the valleys (respectively, peaks) of its complement
\[
\pi^{c}=(n+1-\pi_{1})(n+1-\pi_{2})\cdots(n+1-\pi_{n}),
\]
and vice versa. In Section 3, we complete the enumeration of permutations
with parity restrictions on peaks and valleys using the generalized
run theorem; we show that 
\[
(1+x)\frac{2+2\cosh(\sqrt{2}x)+\sqrt{2}x\sinh(\sqrt{2}x)}{2+2\cosh(\sqrt{2}x)-\sqrt{2}x\sinh(\sqrt{2}x)}
\]
is the exponential generating function for permutations with all peaks
and valleys even, and that 
\[
\frac{2+2\cosh(\sqrt{2}x)+\sqrt{2}(2+x)\sinh(\sqrt{2}x)}{2+2\cosh(\sqrt{2}x)-\sqrt{2}x\sinh(\sqrt{2}x)}
\]
is the exponential generating function for permutations with all peaks
and valleys odd. We give several identities involving the terms of
these generating functions which are proved combinatorially.

Finally, in Section 4 we use the generalized run theorem to find formulae
for bivariate generating functions of the form 
\[
\sum_{n=0}^{\infty}\sum_{\pi\in\mathfrak{S}_{n}}t^{\st(\pi)}\frac{x^{n}}{n!}
\]
for many permutation statistics $\st$ which are expressible in terms
of increasing runs. These statistics include variations of peaks and
valleys, ``double ascents'' and ``double descents'', and ``biruns''
and ``up-down runs''. Although equivalent formulae for the generating
functions for some of these statistics have been discovered already
using other methods, the generalized run theorem provides a straightforward,
systematic method for obtaining these generating functions.

\section{The run theorem}

\subsection{Basic definitions}

Here we give the definitions for some of the basic concepts of this
paper, including descent compositions and noncommutative symmetric
functions.

Given a permutation $\pi$ in $\mathfrak{S}_{n}$, we say that $i\in[n-1]$
is a \textit{descent} of $\pi$ if $\pi_{i}>\pi_{i+1}$. Then, increasing
runs of $\pi$ can be characterized as maximal consecutive subsequences
of $\pi$ containing no descents. Let us call an increasing run \textit{short}
if it has length 1, and \textit{long} if it has length at least 2.

The number of increasing runs of a nonempty permutation is one more
than its number of descents; in fact, the lengths of the increasing
runs determine the descents, and vice versa. Given a subset $S\subseteq[n-1]$
with elements $s_{1}<s_{2}<\cdots<s_{j}$, let $C(S)$ be the composition
$(s_{1},s_{2}-s_{1},\dots,s_{j}-s_{j-1},n-s_{j})$ of $n$, and given
a composition $L=(L_{1},L_{2},\dots,L_{k})$, let $D(L)=\{L_{1},L_{1}+L_{2},\dots,L_{1}+\cdots+L_{k-1}\}$
be the corresponding subset of $[n-1]$. Then, $C$ and $D$ are inverse
bijections. If $\pi$ is a permutation with descent set $S$, we call
$C(S)$ the \textit{descent composition} of $\pi$, which give the
lengths of the increasing runs of $\pi$. Applying $D$ to the descent
composition of a permutation gives its descent set.

The definitions and properties of descents, increasing runs, and descent
compositions extend naturally to words in the free monoid on any totally
ordered alphabet such as $[n]$ or $\mathbb{P}$. Note that increasing
runs in words are allowed to be weakly increasing, whereas increasing
runs in permutations are necessarily strictly increasing since no
letters repeat.

Throughout this section, fix a field $F$ of characteristic zero.
(We can take $F$ to be $\mathbb{C}$ in subsequent sections.) Then
$F\langle\langle X_{1},X_{2},\dots\rangle\rangle$ is the $F$-algebra
of formal power series in countably many noncommuting indeterminates
$X_{1},X_{2},\dots$. Consider the elements 
\[
\mathbf{h}_{n}\coloneqq\sum_{i_{1}\leq\cdots\leq i_{n}}X_{i_{1}}X_{i_{2}}\cdots X_{i_{n}}
\]
of $F\langle\langle X_{1},X_{2},\dots\rangle\rangle$, which are noncommutative
versions of the complete symmetric functions $h_{n}$. Note that $\mathbf{h}_{n}$
is the noncommutative generating function for weakly increasing words
in $\mathbb{P}^{*}$ of length $n$. For example, the weakly increasing
word $13449$ is encoded by $X_{1}X_{3}X_{4}^{2}X_{9}$, which appears
in $\mathbf{h}_{5}$. For a composition $L=(L_{1},\dots,L_{k})$,
we let $\mathbf{h}_{L}=\mathbf{h}_{L_{1}}\cdots\mathbf{h}_{L_{k}}$.
Then the $F$-algebra generated by the elements $\mathbf{h}_{L}$
is the algebra \textbf{Sym} of \textit{noncommutative symmetric functions}
with coefficients in $F$, which is a subalgebra of $F\langle\langle X_{1},X_{2},\dots\rangle\rangle$.

Next, let $\mathbf{Sym}_{n}$ be the vector space of noncommutative
symmetric functions homogeneous of degree $n$, so $\mathbf{Sym}_{n}$
is spanned by $\{\mathbf{h}_{L}\}_{L\models n}$ where $L\models n$
indicates that $L$ is a composition of $n$, and \textbf{Sym} is
a graded $F$-algebra with 
\[
\mathbf{Sym}=\bigoplus_{n=0}^{\infty}\mathbf{Sym}_{n}.
\]

For any composition $L=(L_{1},\dots,L_{k})$, we also define 
\[
\mathbf{r}_{L}=\sum_{L}X_{i_{1}}X_{i_{2}}\cdots X_{i_{n}}
\]
where the sum is over all $(i_{1},\dots,i_{n})$ satisfying 
\[
\underset{L_{1}}{\underbrace{i_{1}\leq\cdots\leq i_{L_{1}}}}>\underset{L_{2}}{\underbrace{i_{L_{1}+1}\leq\cdots\leq i_{L_{1}+L_{2}}}}>\cdots>\underset{L_{k}}{\underbrace{i_{L_{1}+\cdots+L_{k-1}+1}\leq\cdots\leq i_{n}}}.
\]
Then, $\mathbf{r}_{L}$ is the noncommutative generating function
for words in $\mathbb{P}^{*}$ with descent composition $L$. These
$\mathbf{r}_{L}$ are noncommutative versions of the ribbon Schur
functions $r_{L}$.

It can be shown that the noncommutative symmetric functions $\mathbf{h}_{L}$
and $\mathbf{r}_{L}$ can be expressed in terms of each other via
inclusion-exclusion, and that both $\{\mathbf{h}_{L}\}_{L\models n}$
and $\{\mathbf{r}_{L}\}_{L\models n}$ are bases for the vector space
$\mathbf{Sym}_{n}$; see \cite[Section 4.1]{Gessel2014}.

For more on the basic theory of noncommutative symmetric functions,
see \cite{ncsf1}.

\subsection{The original run theorem}

For the remainder of this section, let $A$ be a unital $F$-algebra
of characteristic zero. We are now ready to state the run theorem,
a reciprocity formula involving the noncommutative symmetric functions
$\mathbf{h}_{n}$ and $\mathbf{r}_{L}$ defined in the previous subsection.
\begin{thm}[Run Theorem]
\label{t-runs} Let $L=(L_{1},\dots,L_{k})$ be a composition, $\{w_{1},w_{2},\dots\}$
a set of weights from $A$, and write $w_{L}=w_{L_{1}}w_{L_{2}}\cdots w_{L_{k}}$.
Then, the noncommutative generating function for words in $\mathbb{P}^{*}$,
in which each word with descent composition $L$ is weighted $w_{L}$,
is
\[
\sum_{L}w_{L}\mathbf{r}_{L}=\bigg(\sum_{n=0}^{\infty}v_{n}\mathbf{h}_{n}\bigg)^{-1}
\]
where the sum on the left is over all compositions $L$ and the $v_{n}$
are defined by
\begin{align*}
\sum_{n=0}^{\infty}v_{n}x^{n} & =\bigg(\sum_{k=0}^{\infty}w_{k}x^{k}\bigg)^{-1}
\end{align*}
with $w_{0}=1$ (the unity element of $A$).
\end{thm}
This theorem appeared in its original form as Theorem 5.2 of Gessel
\cite{gessel-thesis}, and is similar to Theorem 4.1 of Jackson and
Aleliunas \cite{MR0450080} and Theorem 4.2.3 of Goulden and Jackson
\cite{MR702512}. Gessel's statement of the theorem does not explicitly
use noncommutative symmetric functions, which were not formally defined
until 1995 in the seminal paper \cite{ncsf1} of Gelfand et al. However,
Gessel and Zhuang \cite[Theorem 11]{Gessel2014} restated the run
theorem using noncommutative symmetric functions and gave a different
proof of the result. 

Both previous versions of the run theorem---Theorem 5.2 of Gessel
\cite{gessel-thesis} and Theorem 11 of Gessel and Zhuang \cite{Gessel2014}---stated
that the weights are to commute with each other, but the proof in
\cite{Gessel2014} does not actually use this condition. Hence, we
allow our algebra $A$ to be commutative or noncommutative.%
\footnote{We do, however, require that the weights commute with noncommutative
symmetric functions. Formally, this means that we are working in the
tensor product algebra $A\otimes_{F}\mathbf{Sym}$.%
} Although we can simply set $A=\mathbb{C}$ in most applications,
the fact that we can take $A$ to be noncommutative is pivotal to
our proof of the generalized run theorem in the next subsection.

We call this result the run theorem because it can be used to obtain
many generating functions which count words and permutations with
various restrictions on the lengths of runs (see \cite[Section 5.2]{gessel-thesis}
and \cite[Section 4.3]{Gessel2014}). In fact, Gessel and Zhuang \cite{Gessel2014}
used this result to explain the similarity between David and Barton's
result (\ref{e-davidbarton}) and the generating function (\ref{e-gz2014})
for permutations with all peaks odd and all valleys even. We briefly
summarize the argument below.

By taking $w_{i}=1$ for $i<m$ and $w_{i}=0$ for $i\geq m$, the
run theorem gives 
\begin{equation}
\sum_{L}\mathbf{r}_{L}=\bigg(\sum_{n=0}^{\infty}(\mathbf{h}_{mn}-\mathbf{h}_{mn+1})\bigg)^{-1}\label{e-wr3}
\end{equation}
where the sum on the left is over all compositions $L$ with all parts
less than $m$. Hence, (\ref{e-wr3}) counts words with every increasing
run having length less than $m$. 

Next, define the homomorphism $\Phi:\mathbf{Sym}\rightarrow F[[x]]$
by $\Phi(\mathbf{h}_{n})=x^{n}/n!$. Then, Gessel and Zhuang \cite[Lemma 9]{Gessel2014}
showed that $\Phi(\mathbf{r}_{L})=\beta(L)x^{n}/n!$, where $\beta(L)$
is the number of permutations with descent composition $L$. Applying
$\Phi$ to (\ref{e-wr3}) then gives 
\[
\left(1-x+\frac{x^{m}}{m!}-\frac{x^{m+1}}{(2m+1)!}+\frac{x^{2m}}{(2m)!}-\frac{x^{2m+1}}{(2m+1)!}+\cdots\right)^{-1}
\]
as the exponential generating function for permutations with every
increasing run having length less than $m$. Setting $m=3$ gives
(\ref{e-davidbarton}).

To obtain (\ref{e-gz2014}), we define a variant of descents and increasing
runs as follows. Given a permutation $\pi$ in $\mathfrak{S}_{n}$,
we say that $i\in[n-1]$ is an \textit{alternating descent} of $\pi$
if $i$ is odd and $\pi_{i}>\pi_{i+1}$ or if $i$ is even and $\pi_{i}<\pi_{i+1}$.
Similarly, an \textit{alternating run} of $\pi$ is a maximal consecutive
subsequence of $\pi$ containing no alternating descents, and we can
also define alternating descent sets and alternating descent compositions
analogously.

We say that $\pi=\pi_{1}\pi_{2}\cdots\pi_{n}$ is an \textit{alternating
permutation} if $\pi_{1}>\pi_{2}<\pi_{3}<\pi_{4}<\cdots$ and that
$\pi$ is \textit{reverse-alternating} if $\pi_{1}<\pi_{2}>\pi_{3}<\pi_{4}>\cdots$.
Alternating permutations are in bijection with reverse-alternating
permutations by complementation, and it is well known that the number
of alternating $n$-permutations is the Euler number $E_{n}$. Incidentally,
an alternating run starting at an even position is an alternating
permutation and an alternating run starting at an odd position is
a reverse-alternating permutation, which hints at a close connection
between these alternating analogues of descents and runs with the
Euler numbers appearing in the generating function (\ref{e-gz2014}).

Define another homomorphism $\hat{\Phi}:\mathbf{Sym}\rightarrow F[[x]]$
by $\hat{\Phi}(\mathbf{h}_{n})=E_{n}x^{n}/n!$. Then, $\hat{\Phi}(\mathbf{r}_{L})=\hat{\beta}(L)x^{n}/n!$
\cite[Lemma 10]{Gessel2014}, where $\hat{\beta}(L)$ is the number
of permutations with alternating descent composition $L$. Applying
$\hat{\Phi}$ to (\ref{e-wr3}) gives (\ref{e-gz2014}) as the exponential
generating function for permutations with every alternating run having
length less than 3, and it is easy to show that these are precisely
the permutations with all peaks odd and all valleys even.

In general, if applying $\Phi$ to an identity obtained by assigning
weights to the run theorem yields a generating function counting permutations
with certain restrictions on increasing runs, then applying $\hat{\Phi}$
to the same identity yields a generating function counting permutations
with the same restrictions on alternating runs.

\subsection{The generalized run theorem}

Suppose that $G$ is a digraph on the vertex set $[m]$, where each
arc $(i,j)$ is assigned a nonempty subset $P_{i,j}$ of $\mathbb{P}$,
and let $P$ be the set of all pairs $(a,e)$ where $e=(i,j)$ is
an arc of $G$ and $a\in P_{i,j}$. In addition, let $\overrightarrow{P^{*}}\subseteq P^{*}$
be the subset of all sequences $\alpha=(a_{1},e_{1})(a_{2},e_{2})\cdots(a_{n},e_{n})$
where $e_{1}e_{2}\cdots e_{n}$ is a walk in $G$. Given $\alpha=(a_{1},e_{1})(a_{2},e_{2})\cdots(a_{n},e_{n})$
in $\overrightarrow{P^{*}}$, let $\rho(\alpha)=(a_{1},a_{2},\dots,a_{n})$,
and let $E(\alpha)=(i,j)$ where $i$ and $j$ are the initial and
terminal vertices, respectively, of the walk $e_{1}e_{2}\cdots e_{n}$.

We call this construction $(G,P)$ a \textit{run network} if for all
nonempty $\alpha,\beta\in\overrightarrow{P^{*}}$, if $\rho(\alpha)=\rho(\beta)$
and $E(\alpha)=E(\beta)$ then $\alpha=\beta$. That is, the same
tuple $(a_{1},a_{2},\dots,a_{n})$ cannot be obtained by traversing
two different walks with the same initial and terminal vertices. Given
a run network $(G,P)$, suppose that we want to count words in $\mathbb{P}^{*}$
with descent composition $L$ whose parts are given by a walk in $G$,
with various weights attached. This can be done using the following
generalization of the run theorem.
\begin{thm}[Generalized Run Theorem]
\label{t-runsm} Suppose that $G$ is a digraph on $[m]$ and that
$(G,P)$ is a run network on $\mathbb{P}^{*}$. Let $\{w_{a}^{(i.j)}:(a,(i,j))\in P\}$
be a set of weights from $A$, with $w_{a}^{(i,j)}=0$ if $(a,(i,j))\notin P$.
Given a composition $L$ and $1\leq i,j\leq m$, let $w^{(i,j)}(L)=w_{L_{1}}^{e_{1}}\cdots w_{L_{k}}^{e_{k}}$
if there exists \textup{$\alpha=(L_{1},e_{1})\cdots(L_{k},e_{k})\in\overrightarrow{P^{*}}$}
such that \textup{$E(\alpha)=(i,j)$} and $L=\rho(\alpha)$, and let
$w^{(i,j)}(L)=0$ otherwise. Then,
\[
\begin{bmatrix}{\displaystyle \sum_{L}w^{(1,1)}(L)\mathbf{r}_{L}} & \cdots & {\displaystyle \sum_{L}w^{(1,m)}(L)\mathbf{r}_{L}}\\
\vdots & \ddots & \vdots\\
{\displaystyle \sum_{L}w^{(m,1)}(L)\mathbf{r}_{L}} & \cdots & {\displaystyle \sum_{L}w^{(m,m)}(L)\mathbf{r}_{L}}
\end{bmatrix}=\begin{bmatrix}{\displaystyle \sum_{n=0}^{\infty}v_{n}^{(1,1)}\mathbf{h}_{n}} & \cdots & {\displaystyle \sum_{n=0}^{\infty}v_{n}^{(1,m)}\mathbf{h}_{n}}\\
\vdots & \ddots & \vdots\\
{\displaystyle \sum_{n=0}^{\infty}v_{n}^{(m,1)}\mathbf{h}_{n}} & \cdots & {\displaystyle \sum_{n=0}^{\infty}v_{n}^{(m,m)}\mathbf{h}_{n}}
\end{bmatrix}^{-1}
\]
where each sum in the matrix on the left-hand side is over all compositions
$L$ and the $v_{n}^{(i,j)}$ are given by
\[
\begin{bmatrix}{\displaystyle \sum_{n=0}^{\infty}v_{n}^{(1,1)}x^{n}} & \cdots & {\displaystyle \sum_{n=0}^{\infty}v_{n}^{(1,m)}x^{n}}\\
\vdots & \ddots & \vdots\\
{\displaystyle \sum_{n=0}^{\infty}v_{n}^{(m,1)}x^{n}} & \cdots & {\displaystyle \sum_{n=0}^{\infty}v_{n}^{(m,m)}x^{n}}
\end{bmatrix}=\left(I_{m}+\begin{bmatrix}{\displaystyle \sum_{k=1}^{\infty}w_{k}^{(1,1)}x^{k}} & \cdots & {\displaystyle \sum_{k=1}^{\infty}w_{k}^{(1,m)}x^{k}}\\
\vdots & \ddots & \vdots\\
{\displaystyle \sum_{k=1}^{\infty}w_{k}^{(m,1)}x^{k}} & \cdots & {\displaystyle \sum_{k=1}^{\infty}w_{k}^{(m,m)}x^{k}}
\end{bmatrix}\right)^{-1}.
\]
\end{thm}
\begin{proof}
We apply the original run theorem with weights coming from the matrix
algebra $\Mat_{m}(A)$. Set 
\[
w_{k}=\begin{bmatrix}w_{k}^{(1,1)} & \cdots & w_{k}^{(1,m)}\\
\vdots & \ddots & \vdots\\
w_{k}^{(m,1)} & \cdots & w_{k}^{(m,m)}
\end{bmatrix}\quad\mathrm{and}\quad w_{L}=\begin{bmatrix}w^{(1,1)}(L) & \cdots & w^{(1,m)}(L)\\
\vdots & \ddots & \vdots\\
w^{(m,1)}(L) & \cdots & w^{(m,m)}(L)
\end{bmatrix}.
\]
It suffices to verify that if $L=(L_{1},\dots,L_{k})$ is a composition,
then $w_{L}=w_{L_{1}}w_{L_{2}}\cdots w_{L_{k}}$. Indeed, the $(i,j)$th
entry of $w_{L_{1}}w_{L_{2}}\cdots w_{L_{k}}$ is
\[
\sum_{\substack{1\leq p_{1},\dots,p_{k+1}\leq m\\
p_{1}=i,\; p_{k+1}=j
}
}w_{L_{1}}^{(p_{1},p_{2})}w_{L_{2}}^{(p_{2},p_{3})}\cdots w_{L_{k}}^{(p_{k},p_{k+1})},
\]
but at most one of these summands is nonzero because a run network
is defined so that the same descent composition cannot be obtained
by traversing two different walks with the same initial and terminal
vertices. This precisely gives us $w^{(i,j)}(L)$, the $(i,j)$th
entry of $w_{L}$, and thus the theorem is proven.
\end{proof}
For example, suppose that we want to count words having descent compositions
of the form $(2,3,2,3\dots,2,3)$. Then, consider the following run
network:
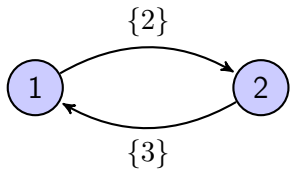
\begin{center}
\begin{tikzpicture}[->,>=stealth',shorten >=1pt,auto,node distance=3cm,   thick,main node/.style={circle,fill=blue!20,draw,font=\sffamily}]

\node[main node] (1) {1};
\node[main node] (2) [right of =1] {2};

\path[every node/.style={font=\sffamily\small}]     
(1) edge [bend left] node {$\{2\}$} (2)    
(2) edge [bend left] node {$\{3\}$} (1);

\end{tikzpicture}
\end{center}

The words that we want to count have descent compositions corresponding
to walks in this digraph beginning and ending at vertex 1. By taking
all nonzero weights to be 1 and applying Theorem \ref{t-runsm}, it
follows that the desired generating function is the $(1,1)$ entry
of the matrix 
\[
\begin{bmatrix}{\displaystyle \sum_{n=0}^{\infty}v_{n}^{(1,1)}\mathbf{h}_{n}}\vphantom{{\displaystyle \frac{\frac{dy_{a}}{dx_{a}}}{\frac{dy_{a}}{dx_{a}}}}} & {\displaystyle \sum_{n=0}^{\infty}v_{n}^{(1,2)}\mathbf{h}_{n}}\\
{\displaystyle \sum_{n=0}^{\infty}v_{n}^{(2,1)}\mathbf{h}_{n}}\vphantom{{\displaystyle \frac{\frac{dy_{a}}{dx_{a}}}{\frac{dy_{a}}{dx_{a}}}}} & {\displaystyle \sum_{n=0}^{\infty}v_{n}^{(2,2)}\mathbf{h}_{n}}
\end{bmatrix}^{-1}
\]
where the $v_{n}^{(i,j)}$ are given by 
\begin{align}
\begin{bmatrix}{\displaystyle \sum_{n=0}^{\infty}v_{n}^{(1,1)}x^{n}}\vphantom{{\displaystyle \frac{\frac{dy_{a}}{dx_{a}}}{\frac{dy_{a}}{dx_{a}}}}} & {\displaystyle \sum_{n=0}^{\infty}v_{n}^{(1,2)}x^{n}}\\
{\displaystyle \sum_{n=0}^{\infty}v_{n}^{(2,1)}x^{n}}\vphantom{{\displaystyle \frac{\frac{dy_{a}}{dx_{a}}}{\frac{dy_{a}}{dx_{a}}}}} & {\displaystyle \sum_{n=0}^{\infty}v_{n}^{(2,2)}x^{n}}
\end{bmatrix} & =\left(I_{2}+\begin{bmatrix}0 & x^{2}\\
x^{3} & 0
\end{bmatrix}\right)^{-1}\nonumber \\
 & =\begin{bmatrix}{\displaystyle \frac{1}{1-x^{5}}}\vphantom{{\displaystyle \frac{\frac{dy}{dx}}{\frac{dy}{dx}}}} & {\displaystyle -\frac{x^{2}}{1-x^{5}}}\\
{\displaystyle -\frac{x^{3}}{1-x^{5}}}\vphantom{{\displaystyle \frac{\frac{dy}{dx}}{\frac{dy}{dx}}}} & {\displaystyle \frac{1}{1-x^{5}}}
\end{bmatrix}\nonumber \\
 & =\begin{bmatrix}{\displaystyle \sum_{n=0}^{\infty}x^{5}}\vphantom{{\displaystyle \frac{\frac{dy_{a}}{dx_{a}}}{\frac{dy_{a}}{dx_{a}}}}} & {\displaystyle -\sum_{n=0}^{\infty}x^{5n+2}}\\
{\displaystyle -\sum_{n=0}^{\infty}x^{5n+3}}\vphantom{{\displaystyle \frac{\frac{dy_{a}}{dx_{a}}}{\frac{dy_{a}}{dx_{a}}}}} & {\displaystyle \sum_{n=0}^{\infty}x^{5}}
\end{bmatrix}.\label{e-egogf}
\end{align}

Therefore, our desired matrix is 
\begin{equation}
\begin{bmatrix}{\displaystyle \sum_{n=0}^{\infty}v_{n}^{(1,1)}\mathbf{h}_{n}}\vphantom{{\displaystyle \frac{\frac{dy_{a}}{dx_{a}}}{\frac{dy_{a}}{dx_{a}}}}} & {\displaystyle \sum_{n=0}^{\infty}v_{n}^{(1,2)}\mathbf{h}_{n}}\\
{\displaystyle \sum_{n=0}^{\infty}v_{n}^{(2,1)}\mathbf{h}_{n}}\vphantom{{\displaystyle \frac{\frac{dy_{a}}{dx_{a}}}{\frac{dy_{a}}{dx_{a}}}}} & {\displaystyle \sum_{n=0}^{\infty}v_{n}^{(2,2)}\mathbf{h}_{n}}
\end{bmatrix}^{-1}=\begin{bmatrix}{\displaystyle \sum_{n=0}^{\infty}\mathbf{h}_{5n}}\vphantom{{\displaystyle \frac{\frac{dy_{a}}{dx_{a}}}{\frac{dy_{a}}{dx_{a}}}}} & {\displaystyle -\sum_{n=0}^{\infty}\mathbf{h}_{5n+2}}\\
{\displaystyle -\sum_{n=0}^{\infty}\mathbf{h}_{5n+3}}\vphantom{{\displaystyle \frac{\frac{dy_{a}}{dx_{a}}}{\frac{dy_{a}}{dx_{a}}}}} & {\displaystyle \sum_{n=0}^{\infty}\mathbf{h}_{5n}}
\end{bmatrix}^{-1}.\label{e-egncsf}
\end{equation}
The homomorphism $\Phi$ defined in the previous subsection induces
a homomorphism on the corresponding matrix algebras, which we also
call $\Phi$ by a slight abuse of notation. Thus, by applying $\Phi$
to (\ref{e-egncsf}), we obtain the matrix 
\begin{equation}
\begin{bmatrix}{\displaystyle \sum_{n=0}^{\infty}\frac{x^{5n}}{(5n)!}}\vphantom{{\displaystyle \frac{\frac{dy_{a}}{dx_{a}}}{\frac{dy_{a}}{dx_{a}}}}} & {\displaystyle -\sum_{n=0}^{\infty}\frac{x^{5n+2}}{(5n+2)!}}\\
{\displaystyle -\sum_{n=0}^{\infty}\frac{x^{5n+3}}{(5n+3)!}}\vphantom{{\displaystyle \frac{\frac{dy_{a}}{dx_{a}}}{\frac{dy_{a}}{dx_{a}}}}} & {\displaystyle \sum_{n=0}^{\infty}\frac{x^{5n}}{(5n)!}}
\end{bmatrix}^{-1},\label{e-egegf}
\end{equation}
whose $(1,1)$ entry is the exponential generating function for permutations
having descent composition of the form $(2,3,2,3\dots,2,3)$. Observe
that (\ref{e-egegf}) is the inverse of the matrix obtained by taking
(\ref{e-egogf}) and converting the ordinary generating functions
to exponential generating functions.

Similarly, we can apply $\hat{\Phi}$ to (\ref{e-egncsf}) and obtain
the analogous result for alternating runs by taking the $(1,1)$ entry
of 
\[
\begin{bmatrix}{\displaystyle \sum_{n=0}^{\infty}E_{n}\frac{x^{5n}}{(5n)!}}\vphantom{{\displaystyle \frac{\frac{dy_{a}}{dx_{a}}}{\frac{dy_{a}}{dx_{a}}}}} & {\displaystyle -\sum_{n=0}^{\infty}E_{n+2}\frac{x^{5n+2}}{(5n+2)!}}\\
{\displaystyle -\sum_{n=0}^{\infty}E_{n+3}\frac{x^{5n+3}}{(5n+3)!}}\vphantom{{\displaystyle \frac{\frac{dy_{a}}{dx_{a}}}{\frac{dy_{a}}{dx_{a}}}}} & {\displaystyle \sum_{n=0}^{\infty}E_{n}\frac{x^{5n}}{(5n)!}}
\end{bmatrix}^{-1}.
\]

Finally, we note that the original run theorem can be retrieved from
the generalized run theorem by using the run network with one vertex
and a loop to which the entire set $\mathbb{P}$ is assigned. Hence,
Theorem \ref{t-runsm} is indeed a generalization of Theorem \ref{t-runs}.

\section{Permutations with parity restrictions on peaks and valleys}

\subsection{The exponential generating functions}

For completeness, we restate Gessel and Zhuang's result for the exponential
generating function for permutations with all peaks odd and all valleys
even.
\begin{thm}[Gessel, Zhuang 2014]
 Let $a_{n}$ be the number of $n$-permutations with all peaks odd
and all valleys even. Then the exponential generating function $A(x)$
for $\{a_{n}\}_{n\geq0}$ is

\begin{align*}
A(x) & =\frac{3\sin\left(\frac{1}{2}x\right)+3\cosh\left(\frac{1}{2}\sqrt{3}x\right)}{3\cos\left(\frac{1}{2}x\right)-\sqrt{3}\sinh\left(\frac{1}{2}\sqrt{3}x\right)}\\
 & =\left(1-E_{1}x+E_{3}\frac{x^{3}}{3!}-E_{4}\frac{x^{4}}{4!}+E_{6}\frac{x^{6}}{6!}-E_{7}\frac{x^{7}}{7!}+\cdots\right)^{-1}
\end{align*}
where $\sum_{k=0}^{\infty}E_{k}x^{k}/k!=\sec x+\tan x$.
\end{thm}
Our main result in this section is the following:
\begin{thm}
\label{t-bcegfs} Let $b_{n}$ be the number of $n$-permutations
with all peaks and valleys even, and let $c_{n}$ be the number of
$n$-permutations with all peaks and valleys odd. Then the exponential
generating functions $B(x)$ for $\{b_{n}\}_{n\geq0}$ and $C(x)$
for $\{c_{n}\}_{n\geq0}$ are 
\[
B(x)=(1+x)\frac{2+2\cosh(\sqrt{2}x)+\sqrt{2}x\sinh(\sqrt{2}x)}{2+2\cosh(\sqrt{2}x)-\sqrt{2}x\sinh(\sqrt{2}x)}
\]
 and 
\[
C(x)=\frac{2+2\cosh(\sqrt{2}x)+\sqrt{2}(2+x)\sinh(\sqrt{2}x)}{2+2\cosh(\sqrt{2}x)-\sqrt{2}x\sinh(\sqrt{2}x)}.
\]

\end{thm}
The first several terms of these sequences are as follows:

\begin{center}
\begin{tabular}{c|ccccccccccccc}
$n$ & 0 & 1 & 2 & 3 & 4 & 5 & 6 & 7 & 8 & 9 & 10 & 11 & 12\tabularnewline
\hline 
$a_{n}$ & 1 & 1 & 2 & 4 & 13 & 50 & 229 & 1238 & 7614 & 52706 & 405581 & 3432022 & 31684445\tabularnewline
\hline 
$b_{n}$ & 1 & 1 & 2 & 6 & 8 & 40 & 84 & 588 & 1632 & 14688 & 51040 & 561440 & 2340480\tabularnewline
\hline 
$c_{n}$ & 1 & 1 & 2 & 2 & 8 & 14 & 84 & 204 & 1632 & 5104 & 51040 & 195040 & 2340480\tabularnewline
\end{tabular}
\par\end{center}

The sequence $\{a_{n}\}_{n\geq0}$ can be found on the OEIS \cite[A246012]{oeis}.

By looking at these numbers, we make the following observations:
\begin{enumerate}
\item $c_{2n}=b_{2n}$ for $n\geq0$;
\item $c_{2n+1}<b_{2n+1}$ for $n\geq1$;
\item $b_{n}<a_{n}$ for $n\geq4$.
\end{enumerate}
The first observation is immediate from the fact that $2n$-permutations
with all peaks and valleys odd are in bijection with $2n$-permutations
with all peaks and valleys even via reflection. That is, if $\pi=\pi_{1}\pi_{2}\cdots\pi_{2n}\in\mathfrak{S}_{2n}$
has all peaks and valleys odd, then its reflection 
\[
\pi^{r}=\pi_{2n}\pi_{2n-1}\cdots\pi_{1}
\]
has all peaks and valleys even, and vice versa. We will show that
the second and third observations are both true at the end of this
section, after computing the generating functions $B(x)$ and $C(x)$.

\subsection{Permutations with all peaks and valleys even}

We first find the exponential generating function $B(x)$ for permutations
with all peaks and valleys even using the generalized run theorem.
To construct a suitable run network, we will need to find increasing
run patterns for these permutations. This can be done in general,
but will depend on whether the permutation begins with an ascent or
descent and ends with an ascent or descent.

Notice that the permutations counted by $B(x)$ which start and end
with ascents are in bijection via complementation with those permutations
which start and end with descents. For the same reason, those starting
with a descent and ending with an ascent are equinumerous with those
starting with an ascent and ending with a descent. Thus, we only need
to consider two cases.

First, let us consider permutations with all peaks and valleys even
which begin and end with ascents. The descent compositions of these
permutations, other than the increasing permutations $12\cdots n$,
are given by walks from vertex 1 to vertex 5 in the following run
network, which we call $(G_{1},P_{1})$:

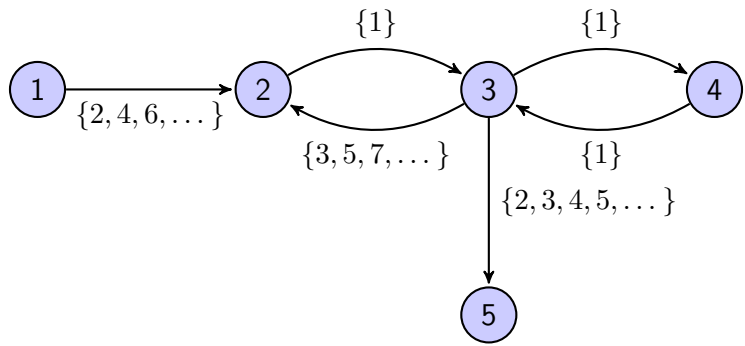
\begin{center}
\begin{tikzpicture}[->,>=stealth',shorten >=1pt,auto,node distance=3cm,   thick,main node/.style={circle,fill=blue!20,draw,font=\sffamily}]

\node[main node] (1) {1};
\node[main node] (2) [right of =1] {2};
\node[main node] (3) [right of =2] {3};
\node[main node] (4) [right of =3] {4};
\node[main node] (5) [below of =3] {5};

\path[every node/.style={font=\sffamily\small}]     
(1) edge [below] node {$\{2, 4, 6, \dots\}$} (2)    
(2) edge [bend left] node {$\{1\}$} (3)
(3) edge [bend left] node {$\{1\}$} (4)
edge [bend left] node {$\{3, 5, 7, \dots\}$} (2)
edge [right] node {$\{2, 3, 4, 5, \dots\}$} (5)
(4) edge [bend left] node {$\{1\}$} (3);

\end{tikzpicture}
\end{center}Indeed, the permutation must begin with an increasing run of even
length before reaching a peak, followed by an odd number of short
increasing runs before reaching a valley.%
\footnote{Recall that an increasing run is called \textit{short} if it has length
1, and it is called \textit{long} if it has length at least 2.%
} Then, going from a valley to a peak corresponds to a long increasing
run of odd length, and once again followed by an odd number of short
increasing runs before reaching another valley. This pattern continues
until the permutation reaches its final valley, and then ends with
a long increasing run.

Let $B_{1}(x)$ denote the exponential generating function for the
permutations corresponding to walks from 1 to 5 in $(G_{1},P_{1})$.
Applying Theorem \ref{t-runsm} with all nonzero weights set equal
to 1 and then applying the homomorphism $\Phi$, we know that $B_{1}(x)$
is the $(1,5)$ entry of 
\[
\begin{bmatrix}{\displaystyle \sum_{n=0}^{\infty}v_{n}^{(1,1)}\frac{x^{n}}{n!}} & \cdots & {\displaystyle \sum_{n=0}^{\infty}v_{n}^{(1,5)}\frac{x^{n}}{n!}}\\
\vdots & \ddots & \vdots\\
{\displaystyle \sum_{n=0}^{\infty}v_{n}^{(5,1)}\frac{x^{n}}{n!}} & \cdots & {\displaystyle \sum_{n=0}^{\infty}v_{n}^{(5,5)}\frac{x^{n}}{n!}}
\end{bmatrix}^{-1},
\]
where the $v_{n}^{(i,j)}$ are given by {\allowdisplaybreaks 
\begin{align*}
\begin{bmatrix}{\displaystyle \sum_{n=0}^{\infty}v_{n}^{(1,1)}x^{n}} & \cdots & {\displaystyle \sum_{n=0}^{\infty}v_{n}^{(1,5)}x^{n}}\\
\vdots & \ddots & \vdots\\
{\displaystyle \sum_{n=0}^{\infty}v_{n}^{(5,1)}x^{n}} & \cdots & {\displaystyle \sum_{n=0}^{\infty}v_{n}^{(5,5)}x^{n}}
\end{bmatrix} & =\left(I_{5}+\begin{bmatrix}0 & \frac{x^{2}}{1-x^{2}} & 0 & 0 & 0\\
0 & 0 & x & 0 & 0\\
0 & \frac{x^{3}}{1-x^{2}} & 0 & x & \frac{x^{2}}{1-x}\\
0 & 0 & x & 0 & 0\\
0 & 0 & 0 & 0 & 0
\end{bmatrix}\right)^{-1}\\
 & =\begin{bmatrix}1 & -\frac{(1-x^{2})x^{2}}{1-2x^{2}} & \frac{x^{3}}{1-2x^{2}} & -\frac{x^{4}}{1-2x^{2}} & -\frac{x^{5}}{(1-2x^{2})(1-x)}\\
0 & \frac{1-2x^{2}+x^{4}}{1-2x^{2}} & \frac{(1-x^{2})x}{1-2x^{2}} & \frac{(1-x^{2})x^{2}}{1-2x^{2}} & \frac{(1+x)x^{3}}{1-2x^{2}}\\
0 & -\frac{x^{3}}{1-2x^{2}} & \frac{1-x^{2}}{1-2x^{2}} & \frac{(1-x^{2})x}{1-2x^{2}} & -\frac{x^{2}(1+x)}{1-2x^{2}}\\
0 & \frac{x^{4}}{1-2x^{2}} & -\frac{(1-x^{2})x}{1-2x^{2}} & -\frac{1-2x^{2}+x^{4}}{1-2x^{2}} & \frac{x^{3}(1+x)}{1-2x^{2}}\\
0 & 0 & 0 & 0 & 1
\end{bmatrix}.
\end{align*}
}

Converting these ordinary generating functions to exponential generating
functions yields the matrix 
\[
\begin{bmatrix}1 & \frac{1-x^{2}-\cosh(\sqrt{2}x)}{4} & -\frac{2x-\sqrt{2}\sinh(\sqrt{2}x)}{4} & \frac{1+x^{2}-\cosh(\sqrt{2}x)}{4} & \frac{3+2x+x^{2}+\sqrt{2}\sinh(\sqrt{2}x)+\cosh(\sqrt{2}x)-e^{x}}{4}\\
0 & \frac{3-x^{2}+\cosh(\sqrt{2}x)}{4} & -\frac{2x+\sqrt{2}\sinh(\sqrt{2}x)}{4} & -\frac{1-x^{2}-\cosh(\sqrt{2}x)}{4} & -\frac{1+2x+x^{2}-\sqrt{2}\sinh(\sqrt{2}x)-\cosh(\sqrt{2}x)}{4}\\
0 & \frac{2x-\sqrt{2}\sinh(\sqrt{2}x)}{4} & \frac{1+\cosh(\sqrt{2}x)}{2} & -\frac{2x+\sqrt{2}\sinh(\sqrt{2}x)}{4} & \frac{2+2x-\sqrt{2}\sinh(\sqrt{2}x)+2\cosh(\sqrt{2}x)}{4}\\
0 & -\frac{1+x^{2}-\cosh(\sqrt{2}x)}{4} & -\frac{2x+\sqrt{2}\sinh(\sqrt{2}x)}{4} & \frac{3+x^{2}+\cosh(\sqrt{2}x)}{4} & -\frac{1+2x+x^{2}-\sqrt{2}\sinh(\sqrt{2}x)-\cosh(\sqrt{2}x)}{4}\\
0 & 0 & 0 & 0 & 1
\end{bmatrix},
\]
whose inverse matrix has $(1,5)$ entry 
\[
B_{1}(x)=\frac{4+4x+x^{2}+(4-x^{2}-4e^{x})\cosh(\sqrt{2}x)+2\sqrt{2}(1+xe^{x})\sinh(\sqrt{2}x)-4e^{x}}{4+4\cosh(\sqrt{2}x)-2\sqrt{2}x\sinh(\sqrt{2}x)}.
\]

Next, we consider permutations with all peaks and valleys even which
begin with a descent and end with an ascent. The increasing runs of
these permutations follow a very similar pattern as before, but it
must begin with an odd number of short increasing runs because the
first letter is a descent rather than an ascent. Therefore, their
descent compositions are given by walks from 1 to 5 in the following
run network, which we call $(G_{2},P_{2})$:

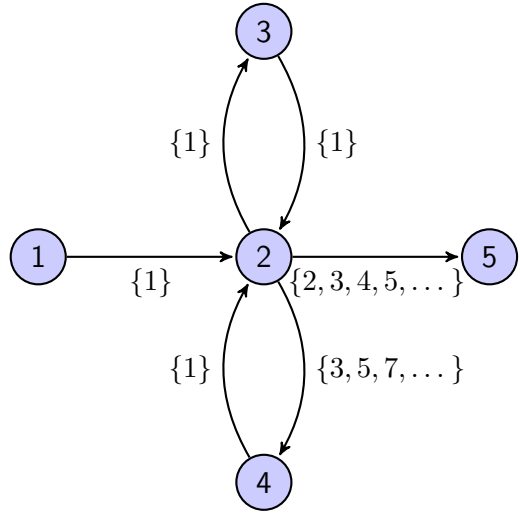
\begin{center}
\begin{tikzpicture}[->,>=stealth',shorten >=1pt,auto,node distance=3cm,   thick,main node/.style={circle,fill=blue!20,draw,font=\sffamily}]

\node[main node] (1) {1};
\node[main node] (2) [right of =1] {2};
\node[main node] (3) [above of =2] {3};
\node[main node] (4) [below of =2] {4};
\node[main node] (5) [right of =2] {5};

\path[every node/.style={font=\sffamily\small}]     
(1) edge [below] node {$\{1\}$} (2)    
(2) edge [bend left] node {$\{1\}$} (3)
edge [bend left] node {$\{3, 5, 7, \dots\}$} (4)
edge [below] node {$\{2, 3, 4, 5, \dots\}$} (5)
(3) edge [bend left] node {$\{1\}$} (2)
(4) edge [bend left] node {$\{1\}$} (2);

\end{tikzpicture}
\end{center}

Repeating the same procedure as before, we start by computing 
\[
\left(I_{5}+\begin{bmatrix}0 & x & 0 & 0 & 0\\
0 & 0 & x & \frac{x^{3}}{1-x^{2}} & \frac{x^{3}}{1-x^{2}}\\
0 & x & 0 & 0 & 0\\
0 & x & x & 0 & 0\\
0 & 0 & 0 & 0 & 0
\end{bmatrix}\right)^{-1}=\begin{bmatrix}1 & -\frac{(1-x^{2})x}{1-2x^{2}} & \frac{(1-x^{2})x^{2}}{1-2x^{2}} & \frac{x^{4}}{1-2x^{2}} & \frac{(1+x)x^{3}}{1-2x^{2}}\\
0 & \frac{1-x^{2}}{1-2x^{2}} & \frac{(1-x^{2})x}{1-2x^{2}} & -\frac{x^{3}}{1-2x^{2}} & -\frac{(1+x)x^{2}}{1-2x^{2}}\\
0 & -\frac{(1-x^{2})x}{1-2x^{2}} & \frac{1-x^{2}-x^{4}}{1-2x^{2}} & \frac{x^{4}}{1-2x^{2}} & \frac{(1+x)x^{3}}{1-2x^{2}}\\
0 & -\frac{(1-x^{2})x}{1-2x^{2}} & \frac{(1-x^{2})x^{2}}{1-2x^{2}} & \frac{(1-x^{2})^{2}}{1-2x^{2}} & \frac{(1+x)x^{3}}{1-2x^{2}}\\
0 & 0 & 0 & 0 & 1
\end{bmatrix},
\]
and then we convert the ordinary generating functions to exponential
generating functions to obtain
\[
\begin{bmatrix}1 & -\frac{2x+\sqrt{2}\sinh(\sqrt{2}x)}{4} & -\frac{1-x^{2}-\cosh(\sqrt{2}x)}{4} & -\frac{1+x^{2}-\cosh(\sqrt{2}x)}{4} & -\frac{1+2x+x^{2}-\sqrt{2}\sinh(\sqrt{2}x)-\cosh(\sqrt{2}x)}{4}\\
0 & \frac{1+\cosh(\sqrt{2}x)}{2} & -\frac{2x+\sqrt{2}\sinh(\sqrt{2}x)}{4} & \frac{2x-\sqrt{2}\sinh(\sqrt{2}x)}{4} & \frac{2+2x-\sqrt{2}\sinh(\sqrt{2}x)-2\cosh(\sqrt{2}x)}{4}\\
0 & -\frac{2x+\sqrt{2}\sinh(\sqrt{2}x)}{4} & \frac{3+x^{2}+\cosh(\sqrt{2}x)}{4} & -\frac{1+x^{2}-\cosh(\sqrt{2}x)}{4} & -\frac{1+2x+x^{2}-\sqrt{2}\sinh(\sqrt{2}x)-\cosh(\sqrt{2}x)}{4}\\
0 & -\frac{2x+\sqrt{2}\sinh(\sqrt{2}x)}{4} & -\frac{1-x^{2}-\cosh(\sqrt{2}x)}{4} & \frac{3-x^{2}+\cosh(\sqrt{2}x)}{4} & -\frac{1+2x+x^{2}-\sqrt{2}\sinh(\sqrt{2}x)-\cosh(\sqrt{2}x)}{4}\\
0 & 0 & 0 & 0 & 1
\end{bmatrix}.
\]
The $(1,5)$ entry of the inverse matrix gives us 
\[
B_{2}(x)=\frac{x^{2}-x(4+x)\cosh(\sqrt{2}x)+2\sqrt{2}\sinh(\sqrt{2}x)}{4+4\cosh(\sqrt{2}x)-2\sqrt{2}x\sinh(\sqrt{2}x)}.
\]

Now we can obtain $B(x)$ by taking $2B_{1}(x)+2B_{2}(x)$, but will
need to add an additional term to account for the increasing and decreasing
permutations which were excluded from the above computations. Since
$e^{x}$ is the exponential generating function for increasing permutations
and also for decreasing permutations, we add $2e^{x}$ but also substract
$x+1$ because the empty permutation and the length 1 permutation
are counted twice by $2e^{x}$.

Therefore, 
\[
B(x)=2B_{1}(x)+2B_{2}(x)+2e^{x}-x-1
\]
which simplifies to 
\[
B(x)=(1+x)\frac{2+2\cosh(\sqrt{2}x)+\sqrt{2}x\sinh(\sqrt{2}x)}{2+2\cosh(\sqrt{2}x)-\sqrt{2}x\sinh(\sqrt{2}x)}.
\]

\subsection{Permutations with all peaks and valleys odd}

Although the exponential generating function $C(x)$ for permutation
with all peaks and valleys odd can be obtained in the same way via
run networks, it is easier to derive it from a combinatorial identity
relating the odd and even terms of the sequence $\{c_{n}\}_{n\geq0}$,
which we do here. 

Notice that the generating function $B(x)$ for permutations with
all peaks and valleys even splits nicely into even and odd parts:
\[
B_{\mathrm{even}}(x)\coloneqq\sum_{n=0}^{\infty}b_{2n}\frac{x^{2n}}{(2n)!}=\frac{2+2\cosh(\sqrt{2}x)+\sqrt{2}x\sinh(\sqrt{2}x)}{2+2\cosh(\sqrt{2}x)-\sqrt{2}x\sinh(\sqrt{2}x)}
\]
and 
\[
B_{\mathrm{odd}}(x)\coloneqq\sum_{n=0}^{\infty}b_{2n+1}\frac{x^{2n+1}}{(2n+1)!}=x\frac{2+2\cosh(\sqrt{2}x)+\sqrt{2}x\sinh(\sqrt{2}x)}{2+2\cosh(\sqrt{2}x)-\sqrt{2}x\sinh(\sqrt{2}x)}.
\]
We immediately deduce from $B_{\mathrm{odd}}(x)=xB_{\mathrm{even}}(x)$
an identity relating the even and odd terms of $\{b_{n}\}_{n\geq0}$. 
\begin{prop}
\label{p-bevenodd} For all $n\geq0$, $b_{2n+1}=(2n+1)b_{2n}$.
\end{prop}
The identity can also be seen from a simple bijection.
\begin{proof}
Let $\pi$ be any $2n$-permutation with all peaks and valleys even,
and pick any $m\in[2n+1]$. Let $\pi^{\prime}$ be the permutation
obtained by replacing the letter $k$ with $k+1$ for every $k\geq m$
in $\pi$, and attaching $m$ to the end of $\pi$. For example, given
$\pi=1432$ and $m=3$, we have $\pi^{\prime}=15423$. Then, $\pi^{\prime}$
is a $(2n+1)$-permutation with all peaks and valleys even. To obtain
$(\pi,m)$ from $\pi^{\prime}$, simply take $m=\pi_{2n+1}^{\prime}$
and apply the reduction map to the word formed by the first $2n$
letters of $\pi^{\prime}$ to get $\pi$.
\end{proof}
Essentially the same bijection gives us an analogous identity for
permutations with all peaks and valleys odd.
\begin{prop}
\label{p-cevenodd} For all $n\geq1$, $c_{2n}=2nc_{2n-1}$.\end{prop}
\begin{proof}
Pick any $(2n-1)$-permutation with all peaks and valleys odd, and
any $m\in[2n]$. Applying the same procedure in the proof of Proposition
\ref{p-bevenodd} yields a $(2n)$-permutation with all peaks and
valleys odd, and we reverse the procedure in the same way to obtain
$\pi$ and $m$.
\end{proof}
Proposition \ref{p-cevenodd}, along with the fact that $b_{2n}=c_{2n}$
for all $n\geq0$, allows us to deduce that {\allowdisplaybreaks
\begin{align*}
C_{\mathrm{odd}}(x) & \coloneqq\sum_{n=1}^{\infty}c_{2n-1}\frac{x^{2n-1}}{(2n-1)!}\\
 & =\sum_{n=1}^{\infty}\frac{c_{2n}}{2n}\frac{x^{2n-1}}{(2n-1)!}\\
 & =\frac{1}{x}\sum_{n=1}^{\infty}b_{2n}\frac{x^{2n}}{(2n)!}\\
 & =\frac{1}{x}\left(\frac{2+2\cosh(\sqrt{2}x)+\sqrt{2}x\sinh(\sqrt{2}x)}{2+2\cosh(\sqrt{2}x)-\sqrt{2}x\sinh(\sqrt{2}x)}-1\right)\\
 & =\frac{2\sqrt{2}\sinh(\sqrt{2}x)}{2+2\cosh(\sqrt{2}x)-\sqrt{2}x\sinh(\sqrt{2}x)}.
\end{align*}
}Furthermore, 
\[
C_{\mathrm{even}}(x)\coloneqq\sum_{n=0}^{\infty}c_{2n}\frac{x^{2n}}{(2n)!}=B_{\mathrm{even}}(x),
\]
so we have 
\begin{align*}
C(x) & =C_{\mathrm{odd}}(x)+C_{\mathrm{even}}(x)\\
 & =\frac{2+2\cosh(\sqrt{2}x)+\sqrt{2}(2+x)\sinh(\sqrt{2}x)}{2+2\cosh(\sqrt{2}x)-\sqrt{2}x\sinh(\sqrt{2}x)},
\end{align*}
which completes the proof of Theorem \ref{t-bcegfs}.

Permutations with all peaks and valleys odd are closely related to
``balanced permutations'', which are defined in terms of standard
skew Young tableaux called ``balanced tableaux''. In fact, balanced
permutations of odd length are precisely permutations of odd length
with all peaks and valleys odd, counted by $\{c_{2n+1}\}_{n\geq0}$.
Gessel and Greene \cite{greene} gave the exponential generating function
for balanced permutations, and by comparing generating functions,
showed that 
\begin{equation}
d_{2n+1}=2^{n}c_{2n+1}\label{e-lacroix}
\end{equation}
for all $n\geq0$, where $d_{n}$ is the number of $n$-permutations
with all valleys odd (and with no parity restrictions on peaks) which
were previously studied by Gessel \cite{Gessel1991}. A bijective
proof of (\ref{e-lacroix}) was later given by La Croix \cite{LaCroix2006}.

\subsection{Two inequalities}

We end this section by verifying two observations made at the beginning
of this section relating the sequences $\{a_{n}\}_{n\geq0},$ $\{b_{n}\}_{n\geq0}$,
and $\{c_{n}\}_{n\geq0}$.
\begin{prop}
For all $n\geq1$, $c_{2n+1}<b_{2n+1}$.\end{prop}
\begin{proof}
Fix $n\geq1$. We provide an injection $\mathtt{shift}$ from the
set of $(2n+1)$-permutations with all peaks and valleys odd to the
set of $(2n+1)$-permutations with all peaks and valleys even.

Let $\pi=\pi_{1}\pi_{2}\cdots\pi_{2n+1}$ have all peaks and valleys
odd, and let $\mathtt{shift}(\pi)=\pi_{2}\cdots\pi_{2n+1}\pi_{1}$.
All of the peaks and valleys of $\pi$ remain peaks and valleys of
$\mathtt{shift}(\pi)$, but their positions are shifted by 1, and
the only possible new valley or peak is given by $\pi_{2n+1}$, which
has an even position in $\mathtt{shift}(\pi)$. Therefore, all of
the peaks and valleys of $\mathtt{shift}(\pi)$ are even, and $\mathtt{shift}$
is injective because it is clearly reversible.

However, $\mathtt{shift}$ is not a surjection. The increasing permutation
$12\cdots(2n)(2n+1)$ has all peaks and valleys even, but $(2n+1)12\cdots(2n)$
does not have all peaks and valleys odd.\end{proof}
\begin{prop}
For all $n\geq4$, $b_{n}<a_{n}$.\end{prop}
\begin{proof}
Fix $n\geq4$. As before, we give a suitable injection, this time
from the set of $n$-permutations with all peaks and valleys even
to the set of $n$-permutations with all valleys even and all peaks
odd.

Let $\pi=\pi_{1}\pi_{2}\cdots\pi_{n}$ have all peaks and valleys
even, and let $\mathtt{pkshift}(\pi)$ be the $n$-permutation obtained
by the following algorithm. We iterate through the letters of $\pi$
from left to right, and for each peak $k<n-1$ of $\pi$, we switch
the letters $\pi_{k}$ and $\pi_{k+1}$. If $n-1$ is a peak of $\pi$,
then we switch $\pi_{n-2}$ and $\pi_{n-1}$. For example, $\pi=287134596$
is mapped to $\mathtt{pkshift}(\pi)=278134956$.

We must show that $\mathtt{pkshift}(\pi)$ has all valleys even and
all peaks odd. Since all peaks and valleys of $\pi$ are even, each
peak must be at least 4 greater than the previous peak, so we can
look at the ``local behavior'' of $\mathtt{pkshift}$ around the peaks
of $\pi$. Suppose that $i<n-1$ is a peak of $\pi$ and is thus even.
Then, 
\[
\pi_{i-2}<\pi_{i-1}<\pi_{i}>\pi_{i+1}>\pi_{i+2}
\]
because $\pi$ has all peaks and valleys even. (It is possible that
$\pi_{i-2}$ does not exist.) Then this segment of $\pi$ is mapped
to 
\[
\pi_{i-2}<\pi_{i-1}\lessgtr\pi_{i+1}<\pi_{i}>\pi_{i+2}
\]
by $\mathtt{pkshift}$ if $i\neq n-1$. If $n-1$ is a peak of $\pi$,
then 
\[
\pi_{n-3}<\pi_{n-2}<\pi_{n-1}>\pi_{n}
\]
is mapped to 
\[
\pi_{n-3}<\pi_{n-1}>\pi_{n-2}\lessgtr\pi_{n}
\]
by $\mathtt{pkshift}$.

These segments may only intersect each other at the leftmost or rightmost
positions, which are not changed, so they are completely compatible.
All of the letters within two positions of the peaks of $\pi$ are
affected in this manner, and none of the other letters are changed. 

In particular, every valley of $\pi$ is unchanged and becomes a valley
of $\mathtt{pkshift}(\pi)$, every peak $i<n-1$ of $\pi$ becomes
a peak $i+1$ of $\mathtt{pkshift}(\pi)$, and if $n-1$ is a peak
of $\pi$ then it becomes a peak $n-2$ of $\mathtt{pkshift}(\pi)$.
Moreover, given a peak $i<n-1$ of $\pi$ and if $\pi_{i-1}>\pi_{i+1}$,
then $i+1$ becomes a valley $i$ of $\mathtt{pkshift}(\pi)$, and
if additionally $\pi_{i-2}$ exists, then $i-1$ becomes a peak of
$\mathtt{pkshift}(\pi)$. If $n-1$ is a peak of $\pi$ and $\pi_{n-2}<\pi_{n}$,
then $n-1$ becomes a valley of $\mathtt{pkshift}(\pi)$. All of these
peaks and valleys are in accordance with the required parity conditions,
and since no other peaks and valleys are introduced, it follows that
$\mathtt{pkshift}$ is a well-defined map into the set of $n$-permutations
with all peaks odd and all valleys even.

This map can easily be reversed. Iterate through the letters of $\mathtt{pkshift}(\pi)$
from right to left. If $n-2$ is a peak of $\mathtt{pkshift}(\pi)$,
then switch $\pi_{n-2}$ and $\pi_{n-1}$. For all other peaks $k$
of $\mathtt{pkshift}(\pi)$, we switch $\pi_{k-1}$ and $\pi_{k}$.
This algorithm eliminates all peaks and valleys that were added and
restores the peaks that were shifted, so it is evident that $\pi$
is recovered from $\mathtt{pkshift}(\pi)$. Hence, $\mathtt{pkshift}$
is an injection. 

As before, $\mathtt{pkshift}$ is not a surjection. Take the $n$-permutation
given by concatenating $2143$ with $56\cdots n$ (the latter is empty
if $n=4$), which has all peaks odd and all valleys even. For $n=4$
and $n>5$, applying the inverse algorithm yields $241356\cdots n$,
which has an odd valley. For $n=5$, the inverse algorithm yields
$21345$, which has all valleys even and no peaks. So, applying $\mathtt{pkshift}$
to $21345$ gives $21345$, which was not the original permutation
$21435$. Thus the result is proven.
\end{proof}

\section{Counting permutations by run-based statistics}

\subsection{Introduction }

Let $\des(\pi)$ and $\altdes(\pi)$ be the number of descents and
alternating descents of a permutation $\pi$, respectively. It was
demonstrated in \cite{Gessel2014} that by setting all weights equal
to $t$ in the original version of the run theorem and applying the
homomorphism $\Phi$, we obtain the famous exponential generating
function 
\[
1+\sum_{n=1}^{\infty}A_{n}(t)\frac{x^{n}}{n!}=\frac{1-t}{1-te^{(1-t)x}}
\]
for the Eulerian polynomials 
\[
A_{n}(t)=\sum_{\pi\in\mathfrak{S}_{n}}t^{\des(\pi)+1},
\]
counting permutations by descents. If we apply $\hat{\Phi}$ instead,
then we get an analogous result 
\[
1+\sum_{n=1}^{\infty}\hat{A}_{n}(t)\frac{x^{n}}{n!}=\frac{1-t}{1-t(\sec((1-t)x)+\tan((1-t)x))}
\]
(previously discovered in an equivalent form by Chebikin \cite[Theorem 4.2]{chebikin})
for the ``alternating Eulerian polynomials'' 
\[
\hat{A}_{n}(t)=\sum_{\pi\in\mathfrak{S}_{n}}t^{\altdes(\pi)+1},
\]
counting permutations by alternating descents.

In general, given a permutation statistic $\st$, we define the $\st$-polynomials
to be 
\[
P_{n}^{\st}(t)\coloneqq\sum_{\pi\in\mathfrak{S}_{n}}t^{\st(\pi)}
\]
for $n\geq0$. So, for instance, the Eulerian polynomials and alternating
Eulerian polynomials correspond to the polynomials for the number
of increasing runs and the number of alternating runs, respectively.
Then the exponential generating functions 
\[
P^{\st}(t,x)\coloneqq\sum_{n=0}^{\infty}P_{n}^{\st}(t)\frac{x^{n}}{n!}
\]
are bivariate generating functions counting permutations by length
and the statistic $\st$.

In this section, we demonstrate how the generating functions $P^{\st}(t,x)$
for many statistics can be obtained by applying the generalized run
theorem on the following run network, which we call $(G,P):$
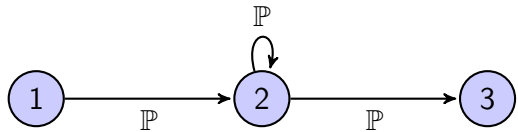
\begin{center}
\begin{tikzpicture}[->,>=stealth',shorten >=1pt,auto,node distance=3cm,   thick,main node/.style={circle,fill=blue!20,draw,font=\sffamily}]

\node[main node] (1) {1};
\node[main node] (2) [right of =1] {2};
\node[main node] (3) [right of =2] {3};

\path[every node/.style={font=\sffamily\small}]     
(1) edge [below] node {$\mathbb{P}$} (2)    
(2) edge [loop above] node {$\mathbb{P}$} (2)
edge [below] node {$\mathbb{P}$} (3);
\end{tikzpicture}
\end{center}The descent compositions of all non-increasing permutations are given
by walks from 1 to 3 in $(G,P)$, but note that we distinguish the
initial run and the final run in this run network. Indeed, the permutation
statistics that we consider in this section are all determined by
the number of non-initial and non-final long increasing runs, as well
as whether the permutation begins and ends with a short increasing
run or a long increasing run. Hence, by assigning appropriate weights
to the letters of $P$ in this run network, the generalized run theorem
yields refined results counting permutations by these statistics.

Not every statistic that we consider requires three vertices in our
run network. For example, only two vertices are required if we only
need to distinguish the initial run or the final run but not both,
and only one vertex is required if we do not need to distinguish the
initial run or the final run. We will still use the 3-vertex run network
$(G,P)$ in the former case, since it eliminates the need for us to
define two 2-vertex run networks (one for distinguished initial runs,
and one for distinguished final runs) and the computation is no more
difficult using a computer algebra system such as Maple. However,
the latter case does not require a run network at all, so we will
simply apply the original version of the run theorem (Theorem \ref{t-runs}).

This general approach can also be used to find multivariate generating
functions giving the joint distribution of two or more of these statistics,
although we do not do this here.

We note that the result of applying the generalized run theorem to
the 3-vertex run network $(G,P)$ is essentially a weighted version
of Theorem 6.12 of Gessel \cite{gessel-thesis}, which gives formulae
for counting words with distinguished initial run and final run, and
is similar to results given by Jackson and Aleliunas \cite[Sections 10-12]{MR0450080}.
See also Goulden and Jackson \cite[Theorem 4.2.19]{MR702512}.

\subsection{Counting permutations by peaks and variations}

The first two statistics that we consider are the number of peaks
and the number of valleys of a permutation. By taking complements,
it is immediate that these two statistics are equidistributed on $\mathfrak{S}_{n}$,
so it suffices to find the bivariate generating function for peaks.
\begin{thm}
Let $\pk(\pi)$ be the number of peaks of $\pi$. Then, 
\[
P^{\pk}(t,x)=\frac{\sqrt{1-t}\cosh(x\sqrt{1-t})}{\sqrt{1-t}\cosh(x\sqrt{1-t})-\sinh(x\sqrt{1-t})}.
\]

\end{thm}
Other equivalent formulae have been found, e.g., by Entringer \cite{Entringer1969}
using differential equations, Mendes and Remmel \cite{Mendes} using
the ``homomorphism method'', and Kitaev \cite{Kitaev2007} using the
notion of partially ordered permutation patterns.

The first ten $\pk$-polynomials are given below, and their coefficients
can be found in the OEIS \cite[A008303]{oeis}.

\begin{center}
\begin{tabular}{c|c|c|c}
$n$ & $P_{n}^{\pk}(t)$ & $n$ & $P_{n}^{\pk}(t)$\tabularnewline
\hline 
0 & $1$ & 5 & $16+88t+16t^{2}$\tabularnewline
1 & $1$ & 6 & $32+416t+272t^{2}$\tabularnewline
2 & $2$ & 7 & $64+1824t+2880t^{2}+272t^{3}$\tabularnewline
3 & $4+2t$ & 8 & $128+7680t+24576t^{2}+7936t^{3}$\tabularnewline
4 & $8+16t$ & 9 & $256+31616t+185856t^{2}+137216t^{3}+7936t^{4}$\tabularnewline
\end{tabular}
\par\end{center}
\begin{proof}
Notice that the number of peaks in a permutation is equal to its number
of non-final long increasing runs, as every non-final long increasing
run necessarily ends with a peak and every peak is at the end of a
non-final long increasing run. So, we want to weight every $k\neq1$
in $P_{1,2}$ and $P_{2,2}$ by $t$ in $(G,P)$. Setting $w_{k}^{(1,2)}=w_{k}^{(2,2)}=t$
for all $k\neq1$ (and setting all other nonzero weights to 1) and
then applying Theorem \ref{t-runsm}, we compute 
\begin{eqnarray*}
\left(I_{3}+\begin{bmatrix}0 & x+\frac{tx^{2}}{1-x} & 0\\
0 & x+\frac{tx^{2}}{1-x} & x+\frac{x^{2}}{1-x}\\
0 & 0 & 0
\end{bmatrix}\right)^{-1} & =\begin{bmatrix}1 & -\frac{(1-(1-t)x)x}{1-(1-t)x^{2}} & \frac{(1-(1-t)x)x^{2}}{(1-(1-t)x^{2})(1-x)}\\
0 & \frac{1-x}{1-(1-t)x^{2}} & -\frac{x}{1-(1-t)x^{2}}\\
0 & 0 & 1
\end{bmatrix} & ,
\end{eqnarray*}
and converting the ordinary generating functions to exponential generating
functions gives 
\begin{equation}
\begin{bmatrix}1 & -1+\cosh(x\sqrt{1-t})-\frac{\sinh(x\sqrt{1-t})}{\sqrt{1-t}} & -1-\frac{\sinh(x\sqrt{1-t})}{\sqrt{1-t}}+e^{x}\\
0 & \cosh(x\sqrt{1-t})-\frac{\sinh(x\sqrt{1-t})}{\sqrt{1-t}} & -\frac{\sinh(x\sqrt{1-t})}{\sqrt{1-t}}\\
0 & 0 & 1
\end{bmatrix}.\label{e-peak}
\end{equation}
Since the increasing permutations have exponential generating function
$e^{x}$ and do not have any peaks, we add $e^{x}$ to the $(1,3)$
entry of the inverse matrix of (\ref{e-peak}) to obtain our desired
generating function 
\[
P^{\pk}(t,x)=\frac{\sqrt{1-t}\cosh(x\sqrt{1-t})}{\sqrt{1-t}\cosh(x\sqrt{1-t})-\sinh(x\sqrt{1-t})}.
\]

\end{proof}
Let us now introduce several variations of peaks and valleys in permutations.
We say that $i$ is a \textit{left peak} of a permutation $\pi=\pi_{1}\pi_{2}\cdots\pi_{n}$
if $i$ is a peak of $\pi$ or if $i=1$ and $\pi_{1}>\pi_{2}$; we
say that $i$ is a \textit{right peak} of $\pi$ if $i$ is a peak
of $\pi$ or if $i=n$ and $\pi_{n-1}<\pi_{n}$; and we say that $i$
is a \textit{left-right peak} of $\pi$ if $i$ is a left peak or
a right peak of $\pi$, or if $\pi=1$ and $i=1$. \textit{Left valleys},
\textit{right valleys}, and \textit{left-right valleys} are defined
analogously.

By taking reverses and complements, it is easy to see that the number
of left peaks, right peaks, left valleys, and right valleys are all
equidistributed on $\mathfrak{S}_{n}$. So, we only need to find the
bivariate generating function for one of these statistics, say, right
peaks. Taking the complement shows that the number of left-right peaks
and left-right valleys are equidistributed as well.

We have the following result:
\begin{thm}
Let $\rpk(\pi)$ and $\lrpk(\pi)$ be the number of right peaks and
left-right peaks of $\pi$, respectively. Then 
\[
P^{\rpk}(t,x)=\frac{\sqrt{1-t}}{\sqrt{1-t}\cosh(x\sqrt{1-t})-\sinh(x\sqrt{1-t})}
\]
and 
\[
P^{\lrpk}(t,x)=\frac{\sqrt{1-t}\cosh(x\sqrt{1-t})-(1-t)\sinh(x\sqrt{1-t})}{\sqrt{1-t}\cosh(x\sqrt{1-t})-\sinh(x\sqrt{1-t})}.
\]

\end{thm}
The first ten $\rpk$-polynomials are given below, and their coefficients
can also be found in the OEIS \cite[A008971]{oeis}. We omit the $\lrpk$-polynomials;
as we shall see shortly, they can be easily characterized in terms
of the $\pk$-polynomials.

\begin{center}
\begin{tabular}{c|c|c|c}
$n$ & $P_{n}^{\rpk}(t)$ & $n$ & $P_{n}^{\rpk}(t)$\tabularnewline
\hline 
0 & $1$ & 5 & $1+58t+61t^{2}$\tabularnewline
1 & $1$ & 6 & $1+179t+479t^{2}+61t^{3}$\tabularnewline
2 & $1+t$ & 7 & $1+543t+3111t^{2}+1385t^{3}$\tabularnewline
3 & $1+5t$ & 8 & $1+1636t+18270t^{2}+19028t^{3}+1385t^{4}$\tabularnewline
4 & $1+18t+5t^{2}$ & 9 & $1+4916t+101166t^{2}+206276t^{3}+50521t^{4}$\tabularnewline
\end{tabular}
\par\end{center}
\begin{proof}
The number of right peaks of a permutation is equal to its total number
of long increasing runs, so we now assign a weight $t$ to every such
run. Thus, setting $w_{k}=t$ for all $k\neq1$ and applying the original
run theorem, we have that 
\[
\left(1+x+\frac{tx^{2}}{1-x}\right)^{-1}=\frac{1-x}{1-(1-t)x^{2}},
\]
whose coefficients have exponential generating function 
\[
\cosh(x\sqrt{1-t})-\frac{\sinh(x\sqrt{1-t})}{\sqrt{1-t}}.
\]
Finally, taking the reciprocal yields 
\begin{align*}
P^{\rpk}(t,x) & =\left(\cosh(x\sqrt{1-t})-\frac{\sinh(x\sqrt{1-t})}{\sqrt{1-t}}\right)^{-1}\\
 & =\frac{\sqrt{1-t}}{\sqrt{1-t}\cosh(x\sqrt{1-t})-\sinh(x\sqrt{1-t})}.
\end{align*}
As for left-right peaks, we claim that the number of left-right peaks
of a permutation $\pi$ is equal to one more than the number of valleys
of $\pi$. The number of valleys of $\pi$ is equal to its number
of non-initial long increasing runs, and every non-initial long increasing
run ends with a left-right peak. In fact, all left-right peaks of
$\pi$ can be found at the end of these runs, but also at the end
of the first increasing run (regardless of length), which accounts
for the difference of 1. Since the number of peaks and number of valleys
are equidistributed, we have that 
\begin{align*}
P^{\lrpk}(t,x) & =tP^{\pk}(t,x)-t+1\\
 & =\frac{\sqrt{1-t}\cosh(x\sqrt{1-t})-(1-t)\sinh(x\sqrt{1-t})}{\sqrt{1-t}\cosh(x\sqrt{1-t})-\sinh(x\sqrt{1-t})}.
\end{align*}
 
\end{proof}
By comparing $P^{\pk}(t,x)$ and $P^{\rpk}(t,x)$, we see the following
relation:
\begin{cor}
The bivariate generating functions $P^{\pk}(t,x)$ and $P^{\rpk}(t,x)$
for the number of peaks and the number of right peaks, respectively,
satisfy 
\[
P^{\pk}(t,x)=P^{\rpk}(t,x)\cosh(x\sqrt{1-t}),
\]
or equivalently, 
\[
P^{\rpk}(t,x)=P^{\pk}(t,x)\sech(x\sqrt{1-t}).
\]

\end{cor}
We do not know of a combinatorial explanation for this fact. Note
that some of the coefficients of 
\[
\cosh(x\sqrt{1-t})=\sum_{n=0}^{\infty}(1-t)^{n}\frac{x^{2n}}{(2n)!}
\]
are negative, so there may be some sort of inclusion-exclusion phenomenon
at play. A combinatorial explanation may also have some connection
to alternating permutations, since 
\[
\sech(x\sqrt{1-t})=\sum_{n=0}^{\infty}E_{2n}(t-1)^{n}\frac{x^{2n}}{(2n)!}
\]
and $E_{2n}$ is the number of alternating permutations of length
$2n$.

\subsection{Counting permutations by double ascents and variations}

Peaks (respectively, valleys) indicate positions in a permutation
where we have a letter that is larger than (respectively, smaller
than) the two surrounding letters, so it is natural to consider the
case when a letter is larger than one surrounding letter but smaller
than the other. Given a permutation $\pi$, we say that $i$ is a
\textit{double ascent} of $\pi$ if $\pi_{i-1}<\pi_{i}<\pi_{i+1}$
and that $i$ is a \textit{double descent} of $\pi$ if $\pi_{i-1}>\pi_{i}>\pi_{i+1}$.
As with peaks and valleys, the number of double ascents and the number
of double descents are equidistributed on $\mathfrak{S}_{n}$, so
we only need to consider double ascents.
\begin{thm}
Let $\dasc(\pi)$ be the number of double ascents of $\pi$. Then,
\[
P^{\dasc}(t,x)=\frac{ue^{\frac{1}{2}(1-t)x}}{u\cosh(\frac{1}{2}ux)-(1+t)\sinh(\frac{1}{2}ux)}
\]
where $u=\sqrt{(t+3)(t-1)}$.
\end{thm}
Below are the first ten $\dasc$-polynomials; see also its OEIS entry
\cite[A162975]{oeis}.

\begin{center}
\begin{tabular}{c|c|c|c}
$n$ & $P_{n}^{\dasc}(t)$ & $n$ & $P_{n}^{\dasc}(t)$\tabularnewline
\hline 
0 & $1$ & 5 & $70+41t+8t^{2}+t^{3}$\tabularnewline
1 & $1$ & 6 & $349+274t+86t^{2}+10t^{3}+t^{4}$\tabularnewline
2 & $2$ & 7 & $2017+2040t+803t^{2}+167t^{3}+12t^{4}+t^{5}$\tabularnewline
3 & $5+t$ & 8 & $13358+16346t+8221t^{2}+2064t^{3}+316t^{4}+14t^{5}+t^{6}$\tabularnewline
4 & $17+6t+t^{2}$ & 9 & $99377+143571t+86214t^{2}+28143t^{3}+4961t^{4}+597t^{5}+16t^{6}+t^{7}$\tabularnewline
\end{tabular}
\par\end{center}
\begin{proof}
It is clear that short increasing runs contribute no double ascents,
and long increasing runs of length $k\geq2$ contribute $k-2$ double
ascents. Thus, we set $w_{k}=t^{k-2}$ for all $k\neq1$ and apply
the original run theorem to obtain 
\[
\left(1+x+\frac{x^{2}}{1-tx}\right)^{-1}=\frac{1-tx}{1+(1-t)(1+x)x},
\]
whose coefficients have exponential generating function 
\[
e^{-\frac{1}{2}(1-t)x}\left(\cosh\Big(\frac{1}{2}ux\Big)-\frac{(1+t)}{u}\sinh\bigg(\frac{1}{2}ux\bigg)\right)
\]
where $u=\sqrt{(t+3)(t-1)}$. Then taking the reciprocal gives us
\begin{align*}
P^{\dasc}(t,x) & =\left(e^{-\frac{1}{2}(1-t)x}\left(\cosh\Big(\frac{1}{2}ux\Big)-\frac{(1+t)}{u}\sinh\Big(\frac{1}{2}ux\Big)\right)\right)^{-1}\\
 & =\frac{ue^{\frac{1}{2}(1-t)x}}{u\cosh(\frac{1}{2}ux)-(1+t)\sinh(\frac{1}{2}ux)}.
\end{align*}

\end{proof}
Now let us consider the variations that we can define for double ascents
and double descents as we did for peaks and valleys. We say that $i$
is a \textit{left double ascent} of a permutation $\pi=\pi_{1}\pi_{2}\cdots\pi_{n}$
if $i$ is a double ascent of $\pi$ or if $i=1$ and $\pi_{1}<\pi_{2}$;
we say that $i$ is a \textit{right double ascent} of $\pi$ if $i$
is a double ascent of $\pi$ or if $i=n$ and $\pi_{n-1}<\pi_{n}$;
and we say that $i$ is a \textit{left-right double ascent} of $\pi$
if $i$ is a left double ascent or a right double ascent of $\pi$,
or if $\pi=1$ and $i=1$. \textit{Left double descents}, \textit{right
double descents}, and \textit{left-right double descents} are defined
analogously. It is evident from taking reverses and complements that
we only need to consider right double ascents and left-right double
ascents.
\begin{thm}
Let $\rdasc(\pi)$ and $\lrdasc(\pi)$ be the number of right double
ascents and left-right double ascents of $\pi$, respectively. Then,
\[
P^{\rdasc}(t,x)=\frac{u\cosh(\frac{1}{2}ux)+(1-t)\sinh(\frac{1}{2}ux)}{u\cosh(\frac{1}{2}ux)-(1+t)\sinh(\frac{1}{2}ux)}
\]
and 
\[
P^{\lrdasc}(t,x)=\frac{ue^{-\frac{1}{2}(1-t)x}}{u\cosh(\frac{1}{2}ux)-(1+t)\sinh(\frac{1}{2}ux)}
\]
where $u=\sqrt{(t+3)(t-1)}$.
\end{thm}
Below are the first ten $\rdasc$- and $\lrdasc$-polynomials. There
is an OEIS entry \cite[A162976]{oeis} for the coefficients of the
$\rdasc$-polynomials, but there does not seem to be one for the $\lrdasc$-polynomials.

\begin{center}
\begin{tabular}{c|c}
$n$ & $P_{n}^{\rdasc}(t)$\tabularnewline
\hline 
0 & $1$\tabularnewline
1 & $1$\tabularnewline
2 & $1+t$\tabularnewline
3 & $3+2t+t^{2}$\tabularnewline
4 & $9+11t+3t^{2}+t^{3}$\tabularnewline
5 & $39+48t+28t^{2}+4t^{3}+t^{4}$\tabularnewline
6 & $189+297t+166t^{2}+62t^{3}+5t^{4}+t^{5}$\tabularnewline
7 & $1107+1902t+1419t^{2}+476t^{3}+129t^{4}+6t^{5}+t^{6}$\tabularnewline
8 & $7281+14391t+11637t^{2}+5507t^{3}+1235t^{4}+261t^{5}+7t^{6}+t^{7}$\tabularnewline
9 & $54351+118044t+111438t^{2}+56400t^{3}+19096t^{4}+3020t^{5}+522t^{6}+8t^{7}+t^{8}$\tabularnewline
\multicolumn{1}{c}{} & \tabularnewline
$n$ & $P_{n}^{\lrdasc}(t)$\tabularnewline
\hline 
0 & $1$\tabularnewline
1 & $t$\tabularnewline
2 & $1+t^{2}$\tabularnewline
3 & $1+4t+t^{3}$\tabularnewline
4 & $6+6t+11t^{2}+t^{4}$\tabularnewline
5 & $19+51t+23t^{2}+26t^{3}+t^{5}$\tabularnewline
6 & $109+212t+269t^{2}+72t^{3}+57t^{4}+t^{6}$\tabularnewline
7 & $588+1571t+1419t^{2}+1140t^{3}+201t^{4}+120t^{5}+t^{7}$\tabularnewline
8 & $4033+10470t+13343t^{2}+7432t^{3}+4272t^{4}+522t^{5}+247t^{6}+t^{8}$\tabularnewline
9 & $29485+87672t+107853t^{2}+87552t^{3}+33683t^{4}+14841t^{5}+1291t^{6}+502t^{7}+t^{9}$\tabularnewline
\end{tabular}
\par\end{center}
\begin{proof}
As before, non-final short increasing runs contribute no right double
ascents, and non-final long increasing runs of length $k\geq2$ contribute
$k-2$ right double ascents. Moreover, if the final increasing run
is of length $k$, then it contributes $k-1$ right double ascents.
So, we take $w_{k}^{(1,2)}=w_{k}^{(2,2)}=t^{k-2}$ for all $k\neq1$
and $w_{k}^{(2,3)}=t^{k-1}$ for all $k$ in the same run network
$(G,P)$ defined earlier, and applying Theorem \ref{t-runsm} gives
\begin{eqnarray*}
\left(I_{3}+\begin{bmatrix}0 & x+\frac{x^{2}}{1-tx} & 0\\
0 & 1+x+\frac{x^{2}}{1-tx} & \frac{x}{1-tx}\\
0 & 0 & 0
\end{bmatrix}\right)^{-1} & =\begin{bmatrix}1 & -\frac{(1+(1-t)x)x}{1+x(1-t)(1+x)} & \frac{(1+(1-t)x)x^{2}}{(1+(1-t)x^{2})(1-tx)}\\
0 & \frac{1-tx}{1+x(1-t)(1+x)} & -\frac{x}{1+x(1-t)(1+x)}\\
0 & 0 & 1
\end{bmatrix} & ,
\end{eqnarray*}
and converting to exponential generating functions gives 
\[
\begin{bmatrix}1 & -1+e^{-\frac{1}{2}(1-t)x}\left(\cosh(\frac{1}{2}ux)-\frac{1+t}{u}\sinh(\frac{1}{2}ux)\right) & -\frac{2}{u}e^{-\frac{1}{2}(1-t)}\sinh(\frac{1}{2}ux)-\frac{1-e^{tx}}{t}\\
0 & e^{-\frac{1}{2}(1-t)x}\left(\cosh(\frac{1}{2}ux)-\frac{1+t}{u}\sinh(\frac{1}{2}ux)\right) & -\frac{2}{u}e^{-\frac{1}{2}(1-t)}\sinh(\frac{1}{2}ux)\\
0 & 0 & 1
\end{bmatrix}.
\]
where $u=\sqrt{(t+3)(t-1)}$. We still need to account for the increasing
permutations, and the increasing permutation of length $n$ has $n-1$
right double ascents. So, we take the $(1,3)$ entry to the inverse
of the above matrix and add to it $(e^{tx}-1)/t+1$ to obtain 
\[
P^{\rdasc}(t,x)=\frac{u\cosh(\frac{1}{2}ux)+(1-t)\sinh(\frac{1}{2}ux)}{u\cosh(\frac{1}{2}ux)-(1+t)\sinh(\frac{1}{2}ux)}.
\]
The computation for left-right double ascents is similar, but we have
to adjust the weights for the initial increasing run. If the initial
increasing run is of length $k$, then it contributes $k-1$ left-right
double ascents; hence, we take $w_{k}^{(2,2)}=t^{k-2}$ for all $k\neq1$
and $w_{k}^{(1,2)}=w_{k}^{(2,3)}=t^{k-1}$ for all $k$. Then the
computation proceeds in the same way, and we add $e^{tx}$ at the
end because the increasing permutation of length $n$ has $n$ left-right
double ascents.
\end{proof}
Comparing our expressions for $P^{\dasc}(t,x)$ and $P^{\lrdasc}(t,x)$
gives the following identity:
\begin{cor}
The bivariate generating functions $P^{\dasc}(t,x)$ and $P^{\lrdasc}(t,x)$
for the number of double ascents and the number of left-right double
ascents, respectively, satisfy
\[
P^{\dasc}(t,x)=e^{(1-t)x}P^{\lrdasc}(t,x).
\]

\end{cor}
We do not know of a combinatorial proof.

\subsection{Counting permutations by biruns and up-down runs}

The number of peaks and valleys in permutations are also closely related
to another statistic that we call the number of \textit{biruns}, which
we define to be maximal consecutive subsequences of length at least
two containing no descents or no ascents, that is, long increasing
runs and ``long decreasing runs''.%
\footnote{Biruns are more commonly called \textit{alternating runs}, but since
the term ``alternating run'' is used for a different concept in this
paper, we use the term ``birun'' which was suggested by Stanley \cite[Section 4]{Stanley2008}.%
} For example, the biruns of $\pi=51378624$ are $51$, $1378$, $862$,
and $24$. We also define an \textit{up-down run} in a permutation
to be an initial short run or a birun. So the up-down runs of $\pi=51378624$
are $5$, $51$, $1378$, $862$, and $24$. 

The notion of biruns has been widely studied in the literature (see,
e.g., \cite[Chapter 1]{Bona2012}), and are closely to related to
the Eulerian polynomials and longest alternating subsequences of permutations.
A connection with the Eulerian polynomials is exemplified by the identity
\[
P_{n}^{\br}(t)=\left(\frac{1+t}{2}\right)^{n-1}(1+w)^{n+1}A_{n}\bigg(\frac{1-w}{1+w}\bigg),\mbox{ for }n\geq2,
\]
where $w=\sqrt{\frac{1-t}{1+t}}$ and $\br(\pi)$ is the number of
biruns of $\pi$ \cite[p.\ 30]{Bona2012}.

We say that a subsequence $\pi_{i_{1}}\cdots\pi_{i_{k}}$ of a permutation
$\pi$ is \textit{alternating} if $\pi_{i_{1}}\cdots\pi_{i_{k}}$
is an alternating permutation. Let $\as(\pi)$ be the length of the
longest alternating subsequence of a permutation $\pi$. Then, for
example, the $n$-permutations $\pi$ with $\as(\pi)=n$ are the length
$n$ alternating permutations. The study of longest alternating subsequences
of permutations was initiated by Stanley \cite{Stanley2008}, who
deduced the bivariate generating function
\begin{equation}
P^{\as}(t,x)=(1-t)\frac{1+v+2te^{vx}+(1-v)e^{2vx}}{1+v-t^{2}+(1-v-t^{2})e^{2vx}},\label{e-as}
\end{equation}
where $v=\sqrt{1-t^{2}}$, and gave the identity

\[
P_{n}^{\as}(t)=\left(\frac{1+t}{2}\right)P_{n}^{\br}(t),\mbox{ for }n\geq2.
\]

We consider up-down runs because the number of up-down runs in a permutation
is equal to the length of its longest alternating subsequence; an
alternating subsequence is obtained by taking the last letter of each
up-down run, and it is easy to see that this is indeed a longest alternating
subsequence. For example, the up-down runs of $\pi=51378624$ are
$5$, $51$, $1378$, $862$, and $24$, so $51824$ is a longest
alternating subsequence of $\pi$, which has length equal to the number
of up-down runs of $\pi$. 

Thus, the study of up-down runs is essentially the study of longest
alternating subsequences through a different lens. In particular,
the generating function $P^{\udr}(t,x)$ that we give below is equivalent
to (\ref{e-as}). Other properties of the $\as$ statistic can be
determined by studying the $\udr$ statistic, and vice versa.
\begin{thm}
Let $\br(\pi)$ and $\udr(\pi)$ be the number of biruns and the number
of up-down runs of $\pi$, respectively. Then, \textup{
\[
P^{\br}(t,x)=\frac{v}{(1+t)^{2}}\cdot\frac{2t+(1+x+t^{2}(1-x))\cosh(vx)-v(1+x)\sinh(vx)}{v\cosh(vx)-\sinh(vx)}
\]
and 
\[
P^{\udr}(t,x)=\frac{v}{1+t}\cdot\frac{t+\cosh(vx)-v\sinh(vx)}{v\cosh(vx)-\sinh(vx)}
\]
where $v=\sqrt{1-t^{2}}$.}
\end{thm}
The first ten $\br$- and $\udr$-polynomials are given below. Also
see their OEIS entries \cite[A059427 and A186370]{oeis}.

\begin{center}
\begin{tabular}{c|c}
$n$ & $P_{n}^{\br}(t)$\tabularnewline
\hline 
0 & $1$\tabularnewline
1 & $1$\tabularnewline
2 & $2t$\tabularnewline
3 & $2t+4t^{2}$\tabularnewline
4 & $2t+12t^{2}+10t^{3}$\tabularnewline
5 & $2t+28t^{2}+58t^{3}+32t^{4}$\tabularnewline
6 & $2t+60t^{2}+236t^{3}+300t^{4}+122t^{5}$\tabularnewline
7 & $2t+124t^{2}+836t^{3}+1852t^{4}+1682t^{5}+544t^{6}$\tabularnewline
8 & $2t+252t^{2}+2766t^{3}+9576t^{4}+14622t^{5}+10332t^{6}+2770t^{7}$\tabularnewline
9 & $2t+508t^{2}+8814t^{3}+45096t^{4}+103326t^{5}+119964t^{6}+69298t^{7}+15872t^{8}$\tabularnewline
\multicolumn{1}{c}{} & \tabularnewline
$n$ & $P_{n}^{\udr}(t)$\tabularnewline
\hline 
0 & $1$\tabularnewline
1 & $t$\tabularnewline
2 & $t+t^{2}$\tabularnewline
3 & $t+3t^{2}+2t^{3}$\tabularnewline
4 & $t+7t^{2}+11t^{3}+5t^{4}$\tabularnewline
5 & $t+15t^{2}+43t^{3}+45t^{4}+16t^{5}$\tabularnewline
6 & $t+31t^{2}+148t^{3}+268t^{4}+211t^{5}+61t^{6}$\tabularnewline
7 & $t+63t^{2}+480t^{3}+1344t^{4}+1767t^{5}+1113t^{6}+272t^{7}$\tabularnewline
8 & $t+127t^{2}+1509t^{3}+6171t^{4}+12099t^{5}+12477t^{6}+6551t^{7}+1385t^{8}$\tabularnewline
9 & $t+255t^{2}+4661t^{3}+26955t^{4}+74211t^{5}+111645t^{6}+94631t^{7}+42585t^{8}+7936t^{9}$\tabularnewline
\end{tabular}
\par\end{center}
\begin{proof}
The number of biruns of a permutation is exactly one more than its
total number of peaks and valleys. To see why this is true, observe
that every non-final long increasing run ends with a peak and that
every non-final long decreasing run ends with a valley, which accounts
for every peak, every valley, and all but the final birun. Thus, in
finding the bivariate generating function for biruns, we can simply
assign weights based on peaks and valleys and make an adjustment at
the end.

Recall that the number of peaks in a permutation is equal to its number
of non-final long increasing runs and that the number of valleys is
equal to its number of non-initial long increasing runs. Hence, using
the run network $(G,P)$ as before, we set $w_{k}^{(1,2)}=w_{k}^{(2,3)}=t$
and $w_{k}^{(2,2)}=t^{2}$ for all $k\neq1$. Then, 
\begin{eqnarray*}
\left(I_{3}+\begin{bmatrix}0 & x+\frac{tx^{2}}{1-x} & 0\\
0 & x+\frac{t^{2}x^{2}}{1-x} & x+\frac{tx^{2}}{1-x}\\
0 & 0 & 0
\end{bmatrix}\right)^{-1} & = & \begin{bmatrix}1 & -\frac{(1-(1-t)x)x}{1-(1-t^{2})x^{2}} & \frac{(1-(1-t)x)^{2}x^{2}}{(1-(1-t^{2})x^{2})(1-x)}\\
0 & \frac{1-x}{1-(1-t^{2})x^{2}} & -\frac{(1-(1-t)x)x}{1-(1-t^{2})x^{2}}\\
0 & 0 & 1
\end{bmatrix},
\end{eqnarray*}
and converting to exponential generating functions gives 
\[
\begin{bmatrix}1 & -\frac{1-\cosh(vx)}{1+t}-\frac{\sinh(vx)}{v} & -1+\left(\frac{1-t}{1+t}\right)x-\frac{2\sinh(vx)}{(1+t)v}+e^{x}\\
0 & \cosh(vx)-\frac{\sinh(vx)}{v} & -\frac{1-\cosh(vx)}{1+t}-\frac{\sinh(vx)}{v}\\
0 & 0 & 1
\end{bmatrix}
\]
where $v=\sqrt{1-t^{2}}$. Finally, we take the $(1,3)$ entry of
the inverse matrix, add $e^{x}$ to account for the increasing permutations,
multiply by $t$, and then add $-tx-t+x+1$. The result is

\[
P^{\br}(t,x)=\frac{v}{(1+t)^{2}}\cdot\frac{2t+(1+x+t^{2}(1-x))\cosh(vx)-v(1+x)\sinh(vx)}{u\cosh(vx)-\sinh(vx)}.
\]

To compute the bivariate generating function $P^{\udr}(t,x)$ for
the number of up-down runs, we use the same weights as before but
also weight initial short runs. That is, we set $w_{k}^{(1,2)}=t$
for all $k$, and set $w_{k}^{(2,2)}=t^{2}$ and $w_{k}^{(2,3)}=t$
for all $k\neq1$. Then the computation is done in the same way, and
at the end we add $e^{x}$, multiply by $t$, and add $-t+1$ to obtain
the desired generating function.
\end{proof}
\textbf{Acknowledgements.} The author would like to thank Ira Gessel---from
whom many important ideas of this paper originated---for numerous
insightful discussions on nearly every aspect of this paper. The results
presented here would not exist without his mentorship. The author
also thanks an anonymous referee, whose helpful comments led to a
much simpler proof of the generalized run theorem.

\bibliographystyle{plain}
\addcontentsline{toc}{section}{\refname}\bibliography{bibliography_firstinitial}

\end{document}